\documentclass[11pt]{article}

\usepackage{algpseudocode}
\usepackage{algorithm}
\usepackage{float}
\usepackage{amsmath}
\usepackage{amssymb}
\usepackage{amsthm}
\usepackage{latexsym}
\usepackage{color}
\usepackage{graphicx}
\usepackage{color}
\usepackage{caption}
\usepackage{subcaption}
\usepackage[normalem]{ulem}
\usepackage{enumitem}
\usepackage{upgreek}
\usepackage{xcolor}
\usepackage{multicol, multirow}
\usepackage[utf8]{inputenc}
\usepackage{eurosym}
\usepackage{hyperref}
\usepackage{float}

\DeclareSymbolFont{calletters}{OMS}{cmsy}{m}{n}
\DeclareSymbolFontAlphabet{\mathcal}{calletters}
    
\makeatletter \@addtoreset{equation}{section}   % reset equation numbering at each section

\newcommand{\bea}{\begin{eqnarray}}
\newcommand{\eea}{\end{eqnarray}}
\newcommand{\bes}{\begin{subequations}}
\newcommand{\ees}{\end{subequations}}
\newcommand{\bgt}{\begin{gather}}
\newcommand{\egt}{\begin{gather}}
\newcommand{\beaa}{\begin{eqnarray*}}
\newcommand{\eeaa}{\end{eqnarray*}}

\newtheorem{Theorem}{Theorem}[section]

\newtheorem{Proposition}[Theorem]{Proposition}

\newtheorem{Assumption}[Theorem]{Assumption}

\newtheorem{Lemma}[Theorem]{Lemma}

\newtheorem{Remark}[Theorem]{Remark}

\newcommand{\ang}[1]{\left[#1\right]}  	  % angular brackets
\newcommand{\brak}[1]{\left(#1\right)}    % round brackets
\newcommand{\crl}[1]{\left\{#1\right\}}   % curly brackets
\newcommand{\tr}[1]{\frac{1}{2}\text{Tr}\ang{#1}}
\newcommand{\Ex}[1]{\E\ang{#1}}
\newcommand{\indi}[1]{\mathbf{1}_{\{#1\}}}

\newcommand{\no}{\noindent}

    \def \st{\text{ s.t }}
 
    \def\Ac{\mathcal{A}}

    \def\Dc{\mathcal{D}}
    \def\Ec{\mathcal{E}}
    \def\Fc{\mathcal{F}}
    \def\Gc{\mathcal{G}}
    \def\Hc{\mathcal{H}}

    \def\Lc{\mathcal{L}}
    
    \def\Nc{\mathcal{N}}
    \def\Oc{\mathcal{O}}

    \def\Uc{\mathcal{U}}
    \def\Vc{\mathcal{V}}

    \def \E{\mathbb{E}}

    \def \P{\mathbb{P}}
    
    \def \R{\mathbb{R}}
    \def \S{\mathbb{S}}
    \def \M{\mathbb{M}}
    \def \N{\mathbb{N}}

    \def \a{\alpha}
    \def \b{\beta}
    \def \v{\nu}
    \def \e{\varepsilon}
    \def \vphi{\varphi}
    \def \k{\kappa}
    \def \t{\theta}
		
    \def \Pas{\text{ } \P\text{ - a.s}}

    \def \tV{^{\t,V}}
    \def \tjw{^{\t,\varpi, j}}
    \def \tjV{^{\t,V, j}}

    \def \tk{_{t_k}}
    \def \tkk{_{t_{k+1}}}

    \def \jtk{^j\tk}
    \def \jtkk{^j\tkk}

    \def \txv{^\nu_{t,x}}

    \def \txvT{\txv(T)}

    \def \muxr{\mu(X\txv(r),\nu_r)}

    \def \sigx{\sigma}
    
    \def \sigxr{\sigma(X\txv(r),\nu_r)}

    \def \tpa{^\a_{t,p}}

    \def \Ptpa{P\tpa}

    \def \Rd{\R^d}
    \def \Rdd{\R^{d+1}}
    \def \Sd{\S^d}

    \def \domd{[0,T]\times \Rdd}
    \def \domint{[0,T)\times \Rd}
    \def \domdint{[0,T)\times \Rdd}

    \def\Ur{{\rm U}}
    \def\ur{{\rm u}}
		
    \def \vhat{\hat{\v}}

    \def \txnhat{(\hat{t}_n,\hat{x}_n)}

    \def \txp{(t,x,p)}

    \def \txo{(t_0,x_0)}

    \def \x{\times}
    \def \1{{\bf 1}}
    \def \vs#1{\vspace{#1pt}}
    \def\eps{\e}
    \def\cl{{\rm cl}}
    \def\inte{{\rm int}}
    \def\vp{\varphi}
        
    \def \eVbT{\varepsilon_{\Vc}^T}
    
    \def \eVbVb{\varepsilon_\Vc^\Vc}

\title{Optimal Control with Expectation Constraint in a Smooth Boundary Case
    }

\author{
Bruno Bouchard
\footnote{CEREMADE, Universit\'e Paris Dauphine - PSL, CNRS.  bouchard@ceremade.dauphine.fr }
\and
Lucas Gnecco Heredia
\footnote{LAMSADE, Miles, Universit\'e Paris Dauphine - PSL. lucas.gnecco-heredia@lamsade.dauphine.fr }
\and 
Ludovic Moreau 
\footnote{Abeille Assurances.  ludovic.moreau@abeille-assurances.fr}
\and 
Kim-Anh Pham 
\footnote{CEREMADE, Universit\'e Paris Dauphine - PSL, and Abeille Assurances.  kim-anh.pham@dauphine.eu }
}

\date{\today}

\begin{document}
	\maketitle
    
% ---------------------------------
%%% Abstract %%%
%----------------------------------

\begin{abstract}
		Motivated by applications in asset and liability management, we study, as in Bouchard et al.~(2010) and Bouchard and Nutz (2014),  a utility maximization problem with expectation constraint. We first consider a uniformly elliptic case in which the endogenous state boundary associated with the constraint in expectation is proved to be smooth.  This allows one to derive a proper  Dirichlet condition for the value function of the optimal control problem on this boundary. We then propose a new truncation argument in the martingale representation of the expectation constraint. This leads to an approximating sequence of auxiliary systems of PDEs for which comparison holds. Convergence to the initial optimal control problem is proved. In the degenerate case, we propose another approximation which consists in adding a small noise term to recover   uniformly ellipticity. Convergence is also proved.  To the best of our knowledge, it is the first time that a full analysis is performed for such control problems, so as to  open the doors to the use of numerical schemes.  We illustrate the usefulness of this approach by presenting a first toy model in the context of asset and liability management. 
		Numerical resolution i is performed using neural networks. It is complemented by an estimation of the numerical error, also performed by using a neural network approach.
\end{abstract}

{\bf Key words:}  Stochastic optimal control, state constraint, stochastic target, discontinuous
viscosity solutions.

{\bf Mathematical subject classifications 2020}: Primary 93E20, 49L25; secondary 68T07.

% ---------------------------------
%%% Introduction %%%
%----------------------------------
\section{Introduction}

% Objective of the paper 
\no As in \cite{bouchard_optimal_2010} and \cite{bouchard_weak_2012}, we study a stochastic control problem under a target constraint in expectation of the form: 
\begin{equation*} 
	V(t,x,p) := \sup \crl{\Ex{F(X\txvT)} \text{ : } \v \in \Uc \st \Ex{G(X\txvT)} \leq p  }
\end{equation*} where $X\txv$ is the  strong solution of the stochastic differential equation 
\begin{equation*}
	X\txv(s) = x + \int_{t}^{s} \muxr dr + \int_{t}^{s} \sigxr dW_r
\end{equation*} with $W$ a $d$-dimensional Brownian motion and $\Uc$ the collection of progressively measurable processes $\v$ taking values in some compact set $U \subset \R^{d'}$, $d,d'\ge 1$.  \\
 In the context of finance and insurance, such problem emerges naturally when portfolio managers seek to maximize final financial results under a constraint over expected risks. This is typically the case in asset and liability management, which motivates this work (see Section \ref{sect_toy example}), this paper being a first step towards the application to a more realistic general asset and liability management model.
\vs5

% Brief discussion on the theoretical problem
Following \cite{bouchard_stochastic_2010}, \cite{bouchard_optimal_2010} and \cite{bouchard_weak_2012}, we use a martingale representation argument for the expectation constraint $\Ex{g(X\txvT)} \leq p $ to convert the above into an optimization problem under state constraint in a regular form. Formally, 
$$
\Ex{G(X\txvT)} \leq p\mbox{ iff } \exists\a\in \Ac\mbox{ s.t. }      \Ptpa\ge w(\cdot,X\txv)\mbox{ on } [t,T]
$$ 
in which 
$$
\Ptpa  := p + \int_{t}^{\cdot} \a_r^\top dW_r
$$
with $\a$ belonging to the set $\Ac$ of  progressively measurable processes such that $P^{\a}_{t,0}$ is a martingale, and 
\begin{equation*}
	w(t,x)  := \inf \{ p \in \R : \;\exists\; (\v,\alpha)\in \Uc\x \Ac\mbox{ s.t. }  G(X\txvT) \leq \Ptpa(T)\}.
\end{equation*}
Then, at a formal level, the geometric dynamic programming principle of \cite{soner_dynamic_2002} leads to
\begin{equation}\label{def V augmented intro}  
	V(t,x,p) := \sup \crl{\Ex{F(X\txvT)}:(\v,\a) \in \Uc\x \Ac \st \Ptpa\ge w(\cdot,X\txv)\mbox{ on $[t,T]$} }.
\end{equation}
Using dynamic programming, one can then deduce a PDE characterization for $(t,x,p)\mapsto V(t,x,p)$. As opposed to traditional optimal control problems with state constraint, see e.g.~\cite{ishii1996new,soner1986optimal,soner1986optimalII} for first order problems and \cite{katsoulakis1994viscosity,lasry1989nonlinear} for second order problems, we cannot impose conditions, such as the traditional interior cone condition,  ensuring that the controlled process  $(X\txv,\Ptpa)$ can be driven in the interior of the domain 
\begin{equation} \label{def domain intro}
    \Dc:=\{ \txp \in \domdint : p\ge w(t,x)\}
\end{equation}
whenever it touches its boundary.  On the contrary, we expect that the boundary $\partial \Dc$ of $\Dc$ is absorbing, meaning that once $(\cdot, X\txv,\Ptpa)$ reaches $\partial \Dc$, it stays on the domain boundary from there onwards until the end of the time horizon. The reason is that $w$ can be alternatively written in
\begin{equation}\label{def w intro}
w(t,x)=\inf_{\v \in \Uc} \Ex{G(X\txvT)}
\end{equation}

so that, upon existence in the above, if 
$\Ptpa(\tau)=w(\tau,X\txv(\tau))$, for some stopping time $\tau$, then the best that can be done is to find a control $\hat \nu_{\tau}$ such that the control $\tilde \nu:=\nu\1_{[t,\tau)}+\hat \nu_{\tau}\1_{[\tau,T]}$ satisfies $\Ptpa(\tau)$ $=$ $\E[G(X^{\tilde \nu}_{t,x}(T))|\Fc_{\tau}]$. If $\tilde \alpha$ satisfies $P_{t,p}^{\tilde \alpha}(T)\ge G(X^{\tilde \nu}_{t,x}(T))$, then $P_{t,p}^{\tilde \alpha}= \E[G(X^{\tilde \nu}_{t,x}(T))|\Fc_{\cdot}]=w(\cdot,X^{\tilde \nu}_{t,x})$ on $[\tau,T]$, because $P_{t,p}^{\tilde \alpha}$ has to be a martingale as well as $w(\cdot,X^{\tilde \nu}_{t,x})$ on $[\tau,T]$ by dynamic programming and optimality of $\hat \nu_{\tau}$. This corresponds to the smooth boundary case of \cite[Section 3.3]{bouchard_optimal_2010} (but in a different context), in which the above reasoning leads to a proper Dirichlet boundary condition for $V$ on $\partial D$,   given by an auxiliary function $\Vc$ which corresponds to the conditional expectation of $F(X\txvT)$ is one plays the best control among those such that  $(X\txv,\Ptpa)$ evolves on $\partial \Dc$ up to $T$. An attempt to deal without an absorbing boundary is done in  \cite[Theorem 3.1]{bouchard_optimal_2010}. However, because of the above argument, their set of test functions needs to be void\footnote{The first author discovered this few years after \cite{bouchard_optimal_2010} was published.}. 
\vs5

% Short description of the proposed solution
In this paper, we first impose smoothness and uniformly ellipticity  conditions to ensure that $w$ is smooth and be able to exploit the ideas in \cite[Section 3.3]{bouchard_optimal_2010}.  In our case, the control $\alpha$ comes from a martingale representation argument, and is therefore unbounded, and the Hamilton-Jacobi-Bellman (HJB) equation for $V$ in the interior of $\Dc$ is associated with a discontinuous operator for which we cannot prove comparison. To circumvent this, we study a sequence of  auxiliary control problems $V^N$ in which $\alpha$ is imposed to be bounded by $N$, and for which the HJB equation admits comparison. Using that $Dw$ is uniformly bounded under our assumptions, we prove that $V^{N}$ converges to $V$ as $N$ goes to $\infty$.  Then, a classical numerical scheme can be used to estimate $V^{N}$, and therefore $V$, by taking $N$ large enough. \\

\no When our uniformly ellipticity condition  does not hold, or when the coefficients are not smooth enough, we explain how the original problem can actually be approximated by a sequence of uniformly elliptic problems with smooth enough coefficients. To the best of our knowledge, it is the first time that such an analysis is carried out, so as to open the doors to the use of numerical schemes.\\

% Brief discussion on the numerical method
As a first steps toward an efficient numerical resolution, we also propose a Deep-Learning-based algorithm which involves sequential estimation of the keys elements of our problem. Firstly, the algorithm begins with approximating $w$ and its associated optimal control which verifies \eqref{def w intro} in order to deduce the viable domain boundary $\partial\Dc$. Secondly, the algorithm continues with the estimation of the value function on $\partial\Dc$, denoted by $\Vc := V(\cdot, w(\cdot))$. The last step is to estimate the optimal control which verifies \eqref{def V augmented intro} and to numerically solve for the value function $V$ using the Hamilton-Jacobi-Bellman equation for the interior of the domain and the Dirichlet conditions on the spatial and temporal boundaries. In particular, to take avantage of the PDE characterization of the problem, we use the Physics-Informed Neural Network (PINN) method which was introduced in \cite{raissi_2019} to be a Deep Learning tool well adapted for solving nonlinear partial differential equations. Furthermore, we also consider several enhancements to the PINN methods as suggested by, for example,  \cite{anagnostopoulos_residual_2024}, \cite{wang_2023}, and \cite{mcclenny_2023} to improve the accuracy and convergence of the PINN training. 
\\

Last but not least, we  showcase the result of our method when applied to an asset and liability management problem in life insurance, that is the initial motivation of this work. As of now, the algorithm is indeed simple and can certainly be improved in various way, but it already provides pretty satisfactory results that would probably be difficult to obtain with standard PDE solvers. The model studied here is very simplistic, but the current work will be completed later on with a paper dedicated to the modeling aspects.

\vs5

% Intro to next paper
{To conclude this introduction, we refer to \cite{pfeiffer2021duality} and its introduction for other approaches (level-set, Pontryagin’s maximum principle, or duality) that can be used to tackle similar problems. Note that the level set approach perfectly matches the original problem only in the case where optimal controls exists while the Pontryagin’s maximum principle and the duality approaches consists in introducing a Lagrangian to tackle the constraint in expectation form. Both approaches are indirect, while ours tackles directly the original problem.
}
\vs5

% Organization
\no The rest of the paper is organized as follows. After defining some notations that will be used all over this paper, we dedicate Section \ref{sect_formulation} to the formal formulation of the stochastic control problem under an expectation constraint and to  the derivation of the associated PDEs. In Section \ref{sect_regularization}, we introduce different approximations that aim at regularizing these PDEs so that numerical methods can be applied appropriately. The deep-leaning-based numerical algorithm is presented in Section \ref{sect_algo}. The example of application within the context of life insurance is provided in Section \ref{sect_toy example}. Last but not least, part of the proofs are collected  in Section \ref{sect_proofs}. 
\vs5

% Notation
\no {\bf General notations :} All over this paper, we adopt the following notations. For any $\k \in \N$, any element of $\R^\k$ is viewed as a column vector. $|\cdot|$ is the Euclidean norm of a vector or a matrix, and $^T$ denotes the transposition. $\M^{k}$ (resp.~$\S^\k$) is the set of (resp.~symmetric) square matrix of dimension $\k \times \k$. Given any smooth function $\vphi : (t,x,y) \in \R_+ \times \R^\k\times \R \to \vphi(t,x,y) \in \R$, let $\partial_t \vphi$ be its derivative with respect to its $t$-variable, and let $D_x\vphi$ (resp. $D_y\vp$) and $D^2_x \vphi$ (resp. $D^2_y\vp$)  be its Jacobian and Hessian matrix with respect to the $x$-component (resp. $y$-component). It $\vphi$ only depends on $(t,x)$, we simply write $D\vp$ and $D^2\vp$ for $D_x\vp$ and $D^2_x\vp$. $\|\vp\|_\infty$ denotes its sup-norm. We use the standard notation $C^{2}(\R^{d})$ (resp.~$C^{2}_{b}(\R^{d})$) for the collection of maps on $\R^{d}$ that are twice continuously differentiable (resp.~and bounded as well as their first and second order derivatives). We extend the above to maps on $B\subset [0,T)\x \R^{d}$ by considering $C^{1,2}(B)$ (resp.~$C^{1,2}_{b}(B)$) defined in an obvious way. For a set $\Oc \subset \R^\k$, $\inte \Oc$ and $\partial \Oc$ denote its interior and boundary, $\cl \Oc$  denotes its closure. Given  $r>0$ and $x \in \R^\k$, $B_r(x)$ is the open ball with radius $r$ and center $x$, $B_r$ is the shorthand for the open ball centered at the origin.  \\

From now on, we consider a finite time horizon $T > 0$ and a complete probability space $(\Omega, \Fc, \P)$ equipped with the $\P$-augmented filtration generated by a $d$-dimensional Brownian motion $W = (W_t)_{0 \leq t \leq T}$ for some $d \geq 1$. All inequalities involving random variables are taken in the $\P$-a.s.~sense.  

% ---------------------------------
%%% Stochastic Control Problem in smooth boundary case %%%
%----------------------------------
\section{The stochastic control problem in the smooth boundary case} \label{sect_formulation}

% Problem statement 
\subsection{Problem statement and alternative formulations} \label{subsect_problem statement}

    % Set-up
\no  Let $\Uc$ be the collection of predictable processes $\v$ taking values in $U$, a compact subset of $\R^{d'}$, for some $d'\ge 1$. For  $t \in [0,T]$, $x \in \Rd$ and $\v \in \Uc$, we define $X\txv$ as the strong solution of the stochastic differential equation 
\begin{equation} \label{dyn_X}
	X\txv(s) = x + \int_{t}^{s}\muxr dr + \int_{t}^{s} \sigxr dW_r, \; t \leq s \leq T
\end{equation} 
in which the map $(\mu,\sigx):\Rd \times U \rightarrow \R \times \M^d$ is bounded, continuous, and Lipschitz continuous in its first variable uniformly in the second one. Time dependence of these parameters could be added in an obvious way. We refrain from doing this for ease of notations. 

\begin{Remark}\label{rem : moments X} It follows from standard arguments that for any $p\ge 1$ there exists $C_{p}>0$ such that, for any $\v\in \Uc$ and $(t,x), (t',x')\in [0,T]\x \R^{d}$,
$$
\Ex{\sup_{[t,T]} |X\txv-X_{t',x'}^{\v}|^{p}}^{\frac1p}\le C_{p}(|x-x'|+|t-t'|^{\frac12}).
$$
\end{Remark}
\vs2

    % Problem statement
\no The aim of this paper is to provide a PDE characterization of the value function of the optimal  control problem under expectation constraint: 
\begin{equation}\label{prob_ori_intro}
	V(t,x,y,p) := \sup_{\nu \in \Uc(t,x,y,p)} \Ex{{\rm F}^\nu_{t,x}} 
\;\mbox{ where }\; \Uc(t,x,y,p):=\{\v \in \Uc : \Ex{{\rm G}^\nu_{t,x,y}} \leq p\}
\end{equation} 
where
\begin{align*}
{\rm F}^\nu_{t,x}:=F(X^{\nu}_{t,x}(T))
+\int_t^T f(X^{\nu}_{t,x}(s),\nu_s)ds,\;{\rm G}^\nu_{t,x,y}&:=G(X^{\nu}_{t,x}(T))
+Y^\nu_{t,x,y}(T)
\end{align*}
with 
$$
Y^\nu_{t,x,y}:=y+\int_t^{\cdot} g(X^{\nu}_{t,x}(s),\nu_s)ds.
$$
 
The reason for introducing the running cost $g$ in the expectation constraint in the form of the process $Y^\nu_{t,x,y}$ is only for ease of exposition. We avoid incorporating $Y$ into the $X$ process because we shall assume later on that this diffusion has an uniformly elliptic quadratic variation coefficient. 

We assume that $F,f,G$ and $g$ are  continuous maps from $\Rd\x U \to \R$ with polynomial growth, which implies that the expectations  are well defined for every $\v \in \Uc$, by Remark \ref{rem : moments X}. 
\vs2

    % Martingale representation 
Let us now follow the steps of \cite{bouchard_optimal_2010,bouchard_stochastic_2010,bouchard_weak_2012} by introducing the collection $\Ac$ of  $\R^d$-valued predictable processes $\a$ such $\int_{0}^{T} |\alpha_{s}|^{2}ds <\infty$ and 
$$
\Ptpa  := p + \int_{t}^{\cdot} \a_r^\top dW_r
$$
is a martingale (a condition that obviously does not depend on $p$) on $[t,T]$, for all $t\le T$. Then, using the martingale representation theorem as in e.g.~\cite[Proposition 3.1]{bouchard_stochastic_2010}, we deduce that 
\begin{align*} 
\Dc:=&\{(t,x,y,p) \in [0,T]\x \R^{d+2}: \Ex{{\rm G}^\nu_{t,x,y}} \leq p \mbox{ for some } \nu \in \Uc\}
\\
=&\{(t,x,y,p) \in [0,T]\x \R^{d+2} : {\rm G}^\nu_{t,x,y}\leq \Ptpa(T) \Pas \mbox{ for some } (\nu,\a) \in \Uc\x \Ac\}. 
\end{align*}
It then follows  that 
\begin{equation}\label{prob_ori_sans borne}
	V(t,x,y,p) := \sup \left\{\Ex{{\rm F}^\nu_{t,x}}:\;(\v,\a) \in \hat \Uc(t,x,y,p)\right\}
\end{equation} 
with 
$$
\hat \Uc(t,x,y,p) :=\{(\v,\a)\in \Uc\x \Ac : {\rm G}^\nu_{t,x,y}\leq \Ptpa(T)   \Pas \},
$$
thus leading to a state constraint for the augmented system $ (X\txv,Y^\nu_{t,x,y},\Ptpa)$. 

        % rewrite the domain boundary to include martingale representation
\begin{Remark}\label{rem : int de D} Define 
\begin{align}\label{eq: def w} 
w(t,x,y):=\inf_{\nu \in \Uc}\Ex{{\rm G}^\nu_{t,x,y}},\;(t,x,y)\in [0,T]\x \R^{d+1}.
\end{align}
One easily checks that $w$ is continuous thanks to the Lipschitz continuity of the coefficients of \eqref{dyn_X} and our continuity and growth assumptions on $G$ and $g$. Then, the parabolic interior of $\Dc$ is given by
\begin{align*} 
\inte_{P} \Dc:=&\{(t,x,y,p)\in [0,T]\x \R^{d+2}: (t,x,y,p)+([0,\eps]\x [-\eps,\eps]^{d+2})\subset \Dc,\; \mbox{for some $\eps>0$}\} \\
=&\{(t,x,y,p)\in [0,T)\x \R^{d+2} : p > w(t,x,y)  \}  
\end{align*}
Upon existence of an optimal control associated with $w$, so that $\Dc$ is actually equal to $\{(t,x,y,p)$ $\in $ $[0,T]\x \R^{d+2} :$  $p \ge w(t,x,y)  \}$, it follows from the geometric dynamic programming principle in \cite{soner_dynamic_2002} that $\Ptpa(T)\ge {\rm G}^\nu_{t,x,y}$ if and only if
$\Ptpa\ge w(\cdot,X^\nu_{t,x},Y^\nu_{t,x,y})$ on $[t,T]$. See also Remark \ref{remark_aborbant_boundary} below. Hence, the state constraint actually applies on the whole path of the processes. When the existence of an optimal control is not guaranteed, the reverse implication requires in general a strict inequality.  
\end{Remark}

        % admissibility of the augmented control 
\begin{Remark}\label{rem : alpha carre inte}
Since ${\rm G}^\nu_{t,x,y}$ is square integrable under our assumptions, if $\Ptpa(T)\ge {\rm G}^\nu_{t,x,y}$ for some $\alpha \in \Ac$, then the same holds for some $\alpha$ satisfying $\E[\int_t^T |\alpha_s|^2ds]<\infty$. Again, this follows from the martingale representation theorem and the fact that $p\ge \E[{\rm G}^\nu_{t,x,y}]$ as $\Ptpa$ is a martingale. One could therefore restrict the set of admissible controls to those that are square integrable.

In \cite[Section 4]{bouchard_weak_2012}, the set of admissible controls in the definition of $V(t,\cdot)$ is defined as controls that are only adapted to the filtration generated by $W_{\cdot \vee t}-W_{t}$. However, as explained in \cite[Remark 5.2]{bouchard2011weak}, a standard randomization argument implies that the value functions coincide.
\end{Remark}

    % Truncation of the problem to remove variable Y
\begin{Remark} \label{rem : V without y} As mentioned above we introduced the process $Y$ corresponding to the running cost only for ease of exposition. We shall however derive the PDEs only for the case $y=0$, using the fact that : 
\begin{equation}\label{eq: def varpi}
\varpi(t,x):=w(t,x,0)=w(t,x,y)-y
\end{equation} 
and
\begin{equation} \label{eq : def V_bar}
V(t,x,y,p)=V(t,x,0,p-y)=:\bar V(t,x,p-y).
\end{equation}
The viability domain associated with $\bar V$ and $\varpi$ is given by 
$$
\bar \Dc:= \{(t,x,p)\in [0,T]\x \R^{d+1} : (t,x,0,p)\in \Dc\}
$$ 
whose parabolic interior is
\begin{align*} 
\inte_{P} \bar \Dc:=&\{(t,x,p)\in [0,T)\x \R^{d+1} : p > \varpi(t,x)  \}.  
\end{align*}

\end{Remark}

% Result on smoothness of the boundary 
\subsection{Smoothness of the boundary and associated Hamilton-Jacobi-Bellman equation}
    % Recall PDE characterization of value function V
It follows from  \cite{bouchard_optimal_2010,bouchard_weak_2012}  that $V$ solves (in the viscosity solution sense) a Hamilton-Jacobi-Bellman equation on $\inte_P \Dc$ associated with the operator
$$
H(t,x,q,q_y,A):= -\sup_{(u,a)\in U\x \R^{d}} \crl{ \hat \mu(x,u)^{\top} q + \frac12 {\rm Tr}\left[(\hat\sigma\hat\sigma^{\top})(x,u,a)A\right]+g(x,u)q_y+f(x,u)}
$$
defined for $(t,x,q,q_y,A)\in [0,T]\x \Rd\x \R^{ {2d}}  \x \R\x \S^{ {2d}}$ with
$$
\hat \mu(\cdot,u):=\left(\begin{array}{c}\mu(\cdot,u)\\ 0_{\Rd}\end{array}\right) \mbox{ and } \hat \sigma(\cdot,u,a):=\left[\begin{array}{c}\sigma(\cdot,u) \\ a^\top\end{array}\right],\;(u,a)\in U\x \R^{ {d}},
$$
in which $0_{\Rd}$ is the vector of $\R^d$ with all entries equal to $0$.
Because of Remark \ref{rem : V without y}, this translates immediatly into a PDE for $\bar V$.
We let $H^{*}$ and $H_{*}$ be the upper- and lower-semicontinuous envelopes of $H$, and $\bar V^{*}$ and $\bar V_{*}$  be the upper- and lower-semicontinuous envelopes of $\bar V$. In the following, we make use of the spatial boundary before $T$:
$$
\partial_{<T} \bar \Dc:=\{(t,x,p)\in \partial \bar \Dc,\;t<T\} = \{(t,x,p) \in [0,T) \x \R^{d+1} : p = \varpi(t,x) \}.
$$
\begin{Theorem}[{\cite[Theorem 4.2]{bouchard_weak_2012}}]\label{thm : pde interior domain} $\bar V_{*}$  is a viscosity supersolution  of 
\begin{align*} 
-\partial_{t}\vp + H^{*}(\cdot,D_{(x,p)}\vp,-D_p \vp,D^{2}_{(x,p)}\vp)&\ge 0 \mbox{ on $\inte_{P}\bar \Dc$},\\
\vp(T,\cdot)&\ge F \mbox{ on $\{(x,p)\in \R^{d+1}: p>G(x)\}$},
\end{align*}
and 
$\bar V^{*}$ is a viscosity  subsolution of
\begin{align*} 
-\partial_{t}\vp + H_{*}(\cdot,D_{(x,p)}\vp,-D_p \vp,D^{2}_{(x,p)}\vp)&\le 0 \mbox{ on $\inte_{P}\bar \Dc\cup \partial_{<T} \bar \Dc$}\\
\vp(T,\cdot)&\le F \mbox{ on $\{(x,p)\in \R^{d+1}: p\ge G(x)\}$}.
\end{align*}
\end{Theorem}

    % Result on smoothness of w   
In order to obtain a suitable  Dirichlet condition on the boundary of $\bar \Dc$, 
we assume that  $G$ is smooth  and that $\sigma$ is uniformly elliptic. We also assume that $g$ is H\"older continuous (and bounded for simplicity).
\begin{Assumption} \label{assum_g_holder}
	$G \in C^2_{b}(\Rd)$, $g$ is bounded, and $D^2G$ and $g(\cdot,u)$ are H\"older continuous on $\Rd$, uniformly in $u\in U$.
\end{Assumption}
\begin{Assumption} \label{assum_sigma_unif_elliptic}
There exists $0<\lambda_{\sigma}\le \Lambda_{\sigma}$ such that 
$$
\lambda_{\sigma}\le z^{\top} (\sigma\sigma^{\top})(x,u)z \le  \Lambda_{\sigma},\;\forall\;(x,u,z)\in \R^{d}\x U\x \partial B_{1}.
$$
\end{Assumption}
Then, standard parabolic theory arguments imply that $\varpi$ 
is a smooth solution of the Hamilton-Jacobi-Bellman equation associated with \eqref{eq: def w}:
\begin{align}\label{eq: HJB w}
0=- \inf_{u\in U} \left(\Lc_{X}^{u} \varpi+g \right) \;\mbox{ on } [0,T)\x \R^{d}, 
 \end{align} 
in which 
$$
\Lc_{X}^{u}\vp:=\partial_{t} \vp+\mu(\cdot,u)^{\top} D\vp +\frac12{\rm Tr}\left[(\sigma\sigma^{\top})(\cdot,u)D^{2}\vp\right], 
$$
for a smooth function $\vp$ and $u\in U$.
\begin{Proposition}\label{prop: w smooth} Let Assumptions \ref{assum_g_holder} and  \ref{assum_sigma_unif_elliptic} hold, then $\varpi$ is continuous, it belongs to $C^{1,2}([0,T)\x \R^{d})\cap C^{0,1}_{b}([0,T]\x \R^{d}) $. $D^2\varpi$ is locally H\"older continuous on $[0,T)\x \R^{d}$. Moreover, $\varpi$ solves  \eqref{eq: HJB w} and satisfies $\varpi(T,\cdot)=G$ on $\Rd$.
\end{Proposition} 
The proof is postponed to Section \ref{sec: smoothness w}. 
\vs5

    % Absorbing boundary 
\begin{Remark} \label{remark_aborbant_boundary}
Upon existence of an optimal control, the smoothness of $\varpi$, and therefore $w$, ensures that the boundary of $ \Dc$ is absorbing. Indeed,  given $(t,x,y,p)\in \Dc$, let $(\nu,\alpha)\in \Uc\x \Ac$ be such that $P^{\alpha}_{t,p}(T)\ge {\rm G}^\nu_{t,x,y}$. Set $\tau:=\inf\{s\ge t: P^{\alpha}_{t,p}(s)=w(s,X^{\nu}_{t,x}(s),Y^{\nu}_{t,x,y}(s))\}\wedge T$. Since $w(\cdot,X^{\nu}_{t,x},Y^{\nu}_{t,x,y})$ is a submartingale, by the dynamic programming principle, while $P^{\alpha}_{t,p}$ is a martingale, the inequality $P^{\alpha}_{t,p}(T)\ge {\rm G}^\nu_{t,x,y}=w(T,X^{\nu}_{t,x}(T),Y^{\nu}_{t,x,y}(T))$ implies 
$$
P^{\alpha}_{t,p}(\theta)=\E[P^{\alpha}_{t,p}(T)|\Fc_{\theta}]\ge \E[w(T,X^{\nu}_{t,x}(T),Y^{\nu}_{t,x,y}(T))|\Fc_{\theta}] \ge w(\theta,X^{\nu}_{t,x}(\theta),Y^{\nu}_{t,x,y}(\theta))
$$
for all stopping time $\theta$ taking values in $[\tau,T]$. But, taking expectation again, 
$$
P^{\alpha}_{t,p}(\tau)=\E[P^{\alpha}_{t,p}(\theta)|\Fc_{\tau}]\ge \E[w(\theta,X^{\nu}_{t,x}(\theta),Y^{\nu}_{t,x,y}(\theta))|\Fc_{\tau}]\ge w(\tau,X^{\nu}_{t,x}(\tau),Y^{\nu}_{t,x,y}(\tau))=P^{\alpha}_{t,p}(\tau)
$$
so that $P^{\alpha}_{t,p}(\theta)=w(\theta,X^{\nu}_{t,x}(\theta),Y^{\nu}_{t,x,y}(\theta))$. This   shows that $w(\cdot,X^{\nu}_{t,x},Y^{\nu}_{t,x,y})$ has to admit a martingale decomposition from $\tau$ on. By \eqref{eq: HJB w}, the above is possible only if $\nu \in \Ur(\cdot, X^{\nu}_{t,x})$ and $\alpha
^\top=D_{x}w(\cdot, X^{\nu}_{t,x}, Y^{\nu}_{t,x,y})\sigma(X^{\nu}_{t,x},\nu)$ $=D\varpi(\cdot, X^{\nu}_{t,x})\sigma(X^{\nu}_{t,x},\nu)$ from $\tau$ on, in which 
$$
\Ur(t',x'):={\rm arg}\min\{ \Lc_{X}^{u} \varpi(t',x')+g(x',u): u\in U\},\;(t',x')\in [0,T)\x \R^{d}.
$$ 
\end{Remark}

From the above Remark,  we expect that the PDE solved by 
$$
\Vc:(t,x,y)\in [0,T]\times \Rd \mapsto  V(t,x,y,w(t,x,y))
$$ can be deduced 
 from the Hamilton-Jacobi-Bellman equation of Theorem \ref{thm : pde interior domain} with controls restricted to $\{(u,a)\in \Ur(t,x)\x\{(\sigma(x,u)^\top D\varpi(t,x)\}\}$ at $(t,x)$.  

     % Assumption on the optimal control on the boundary (in preparation for the result on Vb)
This is the result of Theorem \ref{thm : dirichlet condition} below, in which the supersolution property requires an additional assumption.

\begin{Assumption} \label{assum_selection_unique_opti} For all $(t_{0},x_{0}) \in [0,T)\x \R^{d}$ and  $u_0 \in \Ur\txo$, there exists a $U$-valued continuous map $\hat {\rm u}$ on $[0,T)\x \R^{d}$ and $\eps\in (0,t_{0})$ such that 
\begin{align}\label{eq: def U(t,x)} 
\hat {\rm u}(t_0,x_0)=u_0\;\mbox{ and }\;\hat {\rm u}(t,x)\in \Ur(t,x) \;\mbox{ for all } (t,x)\in [t_{0}-\eps,T)\x \R^{d},
\end{align}
and  the stochastic differential equation 
 \begin{align}\label{def hat X dans Assumptio}
 \hat X_{t,x}(s) = x + \int_{t}^{s}\mu(\hat X_{t,x}(r),\hat {\rm u}(r,\hat X_{t,x}(r))) dr + \int_{t}^{s} \sigma(\hat X_{t,x}(r),\hat {\rm u}(r,\hat X_{t,x}(r))) dW_r, \; t \leq s \le T,
  \end{align}
 admits a unique strong solution, for all $(t,x)\in [t_{0}-\eps,T)\x \R^{d}$.
 \end{Assumption}

\begin{Remark} Given Remark \ref{rem : moments X}, it suffices that $\hat \ur$ is locally Lipschitz in its space variable to ensure that $ \hat X_{t,x}$ is well-defined. It is proved by a standard localization argument.
We refer to e.g. \cite{veretennikov1981strong} for more general sufficient conditions. In particular, existence holds if $\sigma$ does not depend on $u$, $\mu$ is locally bounded, and $\hat {\rm u}$ is only measurable. See also Section \ref{sect_approximation cas non regulier} for possible approximation arguments ensuring that $\hat \Ur(t,x)=\{\hat \ur(t,x)\}$ for all $(t,x)$, in which $\hat \ur$ is continuous or even locally Lipschitz.
\end{Remark}

    % Result for value function at the boundary Vb
In the following, we let 
\begin{align*}
\bar \Vc_{*}&:=\bar V_{*}(\cdot,\varpi)\mbox{ and } \bar \Vc^{*}:=\bar V^{*}(\cdot,\varpi)
 \end{align*}
 be the upper- and lower semicontinuous envelopes of $\bar V$ when reaching the boundary.  
 
  \begin{Remark}\label{rem : Vc decroi en p} Since $p\mapsto \bar V(\cdot,p)$ is non-decreasing, we have 
\begin{align*}
\bar \Vc_{*}(t,x)&=\liminf\{\bar V(t',x',\varpi(t',x')):  [0,T]\x \R^{d}\ni (t',x')\to (t,x)\}\\
\bar\Vc^{*}(t,x)&=\limsup\{\bar V(t',x',\varpi(t',x')+\eps'):  [0,T]\x \R^{d}\x (0,\infty)\ni (t',x',\eps')\to (t,x,0)\}
 \end{align*}
 for all $(t,x)\in [0,T]\x \R^{d}$.
 \end{Remark}
 
 \begin{Theorem}\label{thm : dirichlet condition} Let the conditions of Proposition \ref{prop: w smooth}  hold. Then,  $\bar\Vc^{*}$  is a viscosity  subsolution of 
\begin{align} 
-  \max_{u\in \Ur(t,x)}\left(\Lc_{X}^{u} \vp(t,x)+f(x,u)\right)&=0,\;(t,x)\in [0,T)\x \R^d \label{eq: pde Vc} \\
\vp(T,x)&=F(x), \;x\in \R^d.\label{eq: cond bord Vc}
\end{align}
 If moreover Assumption \ref{assum_selection_unique_opti} holds then $\bar\Vc_{*}$   is a viscosity supersolution of \eqref{eq: pde Vc}-\eqref{eq: cond bord Vc}.
\end{Theorem}
The proof is postponed to Section \ref{sect_proof_PDE_Vc}. 

\begin{Remark}\label{rem : unicite edp cal V}
We refer to \cite{crandall1992user} for sufficient conditions for \eqref{eq: pde Vc}-\eqref{eq: cond bord Vc} to admit a comparison principle and a unique (continuous) solution. This will in particular be the case, in the class of functions with polynomial growth, if ${\rm U}(t,x)$ is a singleton for all $(t,x)\in [0,T)\x \Rd$ and $\hat {\rm u}$ is uniformly Lipschitz  continuous in its second argument on each compact subset of $[0,T)\times \Rd$. This is easily obtained by adapting the traditional proof in \cite{crandall1992user} with standard localization arguments as in e.g.~\cite[Proof of Proposition 4.12]{bouchard2009stochastic}. 

If comparison holds, then $\bar \Vc=\bar \Vc_*=\bar \Vc^*$ and the boundary condition for $\bar V$ on $\partial \bar D= \{(t,x,p) \in [0,T] \x \R^{d+1} : p = \varpi(t,x) \}$ is given by $\bar \Vc$, which complements Theorem \ref{thm : pde interior domain}.
\end{Remark}

% ---------------------------------
%%% Approx. by bounded controls & regular functions %%%
%----------------------------------

\section{Regularization}\label{sect_regularization}

In this section, we explain how  approximations can be used to   reduce the Hamilton-Jacobi-Bellman operator in Theorem \ref{thm : pde interior domain} to a continuous one or  recover the conditions of Assumptions \ref{assum_g_holder},  \ref{assum_sigma_unif_elliptic} and \ref{assum_selection_unique_opti} when they are not satisfied. It aims at ensuring that $\varpi$ is smooth and  recovering PDEs for which comparison holds and numerical schemes can be used.

% Approximation by bounded controls (imposing bounds on the Martingale control)
\subsection{Approximation by bounded controls} \label{sect_approximation control bornes}
    
    % Introduction of the auxiliary problem (where control a is bounded)
The Hamilton-Jacobi-Bellman operator in Theorem \ref{thm : pde interior domain} is not continuous \textit{a priori} because the control $\a$ for the augmented martingale process $M^\alpha$ is not bounded \textit{a priori}. In order to regularize the problem, and  be able to obtain a comparison result (which is important to ensure the convergence of numerical schemes), one can approximate $V$ by a sequence of auxiliary control problems with bounded controls. Without loss of generality (see Remark \ref{rem : V without y}), we set $y=0$ for this entire section, and hence the viability domain is restricted to 
$$
\bar \Dc = \crl{(t,x,p) \in \domd : \exists (\v,\a) \in \Uc \x \Ac \st G^{\v}_{t,x,0} \leq \Ptpa(T) \mbox{ a.s}}
$$ in which the geometric dynamic principle applies upon $(\cdot,X\txv,\Ptpa)$. Recall from Proposition \ref{prop: w smooth} that $D\varpi$ is bounded and that $\sigma$ is bounded by assumption, and define
\begin{align}\label{def : varpi}
    C_{\varpi}:=\sup_{(t,x,u) \in [0,T]\times \Rd\times U}| D\varpi(t,x)\sigma(x,u)|<\infty.
\end{align}

For any  $N \geq C_{\varpi}$, let $\Ac_{N}$ be the collection of elements $\alpha$ in $\Ac$ such that $|\alpha|\le N$  $dt\x d\P$-a.s. 
Following \eqref{prob_ori_sans borne}, we define
\begin{equation}\label{prob_ori_borne par N}
	\bar V^{N}(t,x,p) := \sup\left\{ \Ex{F^{\v}_{t,x}}: (\v,\a) \in \bar \Uc^N(t,x,p)\right\}
\end{equation}
where 
$$
\bar \Uc^N(t,x,p) :=  (\Uc\x \Ac_{N})\cap \hat \Uc(t,x,0,p) .
$$

It is natural to expect that $\bar V^N$ converges to $\bar V$ as $N\to \infty$. This is indeed possible thanks to \eqref{def : varpi}. The convergence stated below just requires to slightly bump the level of the constraint.
        % result : convergence of the auxiliary value function V^N to the original value function V
\begin{Proposition}\label{prop: approx N} Let the conditions of Proposition \ref{prop: w smooth}  and Assumption \ref{assum_selection_unique_opti} hold. Then, 
$
\bar \Uc^N(t,x,p)  \ne \emptyset$ 
 for all $(t,x,p)\in  \bar \Dc$ and $N\ge C_{\varpi}$. If moreover ${\rm U}(t,x)$ is a singleton for all $(t,x)\in [0,T)\x \Rd$, then  $\lim_{\eps \downarrow 0}\lim_{N\to \infty} \bar V^{N}(t,x,p+\eps)=\bar V(t,x,p)$ for all $(t,x,p)\in\inte\bar \Dc$. 
\end{Proposition}
            
\begin{proof} Since $\bar V^N$ is non-decreasing in $N$, it suffices to prove the convergence along a subsequence.
            % proof : optimal control exists
            
{\bf a.} Fix an arbitrary $(t,x,p) \in \bar \Dc$. By Proposition \ref{prop: w smooth}, Assumption \ref{assum_selection_unique_opti} and It\^o's Lemma, there exists an admissible control $(\v,\a) \in \Uc \x \Ac$ such that $(\cdot,X\txv, \Ptpa) \in \bar \Dc$ on $[t,T]$. For  $N \geq C_{\varpi}$, let us define 
$$
\a^N := -N \vee \a \wedge N \mbox{ and } P^N := p +\eps_0 + \int_t^{\cdot} (\a^N_r)^\top dW_r,
$$ 
for some $\eps_0\ge 0$.
By construction, $\a^N \in \Ac^N$. Using Remark \ref{rem : alpha carre inte}, one can assume that $\alpha$ is square integrable so that  $\lim_{N \to \infty} \E[\int_t^T |\a^N_s -\a_s|^2ds]=0$, by dominated convergence, and therefore 
\begin{align}\label{M_converge}
\sup_{[t,T]} |P^N{-\eps_0}-P^\alpha_{t,p}|(\omega)\to 0 \;\mbox{ for all } \omega \notin {\cal N}\mbox{ as $N\to \infty$}
\end{align}
where ${\cal N}$ is a null set and the convergence above is taken along a subsequence if necessary. 

By Assumption \ref{assum_selection_unique_opti}, there is a  continuous map $\hat {\rm u}$ such that $\Lc^{\hat{\rm u}(\cdot)}_X \varpi (\cdot) = 0$ on ${[0,T)\x \Rd}$. By Proposition \ref{prop: w smooth}, the map 
$$
\hat {\rm a} : (t,x) \mapsto   ({D\varpi(t,x)\sigma(x,\hat{\rm u}(t,x))})^\top
$$  
is well-defined on $ {[0,T)\x \Rd}$, and, by definition of $C_{\varpi}$, $|\hat {\rm a}| \leq C_{\varpi} \leq N $ for $N$ large enough. Consider the following stopping times
\begin{align*}
    \tau^N &:= \inf\{t' \geq t : P^N(t') = \varpi(t',X\txv(t'))\}\wedge T  \\
    \tilde \tau &:=  \inf\{t' \geq t : \Ptpa(t') = \varpi(t',X\txv(t'))\} \wedge T\\
    \hat \tau^N&:=\tau^N\wedge \tilde \tau 
\end{align*}
and the control process $(\bar \v^N, \bar \a^N)$ defined by 
\begin{align*}
    \bar \v^N := \v \mathbf{1}_{[t, \hat \tau^N  ]} + \hat{\rm u}(\cdot,X^{\bar \v^N}_{t,x}) \mathbf{1}_{(\hat \tau^N  ,T)} \hspace{20pt} 
    \bar \a^N := \a^N \mathbf{1}_{[t,\hat \tau^N  ]} + \hat{\rm a}(\cdot,X^{\bar \v^N}_{t,x}) \mathbf{1}_{(\hat \tau^N ,T)}. 
\end{align*}
Assumption \ref{assum_selection_unique_opti} ensures that it is well-defined. 
 Note that $(\bar \v^N, \bar \a^N) \in \Uc \x \Ac^N$ and that, by definition of $(\hat{\rm u},\hat{\rm a})$, Proposition \ref{prop: w smooth}  and It\^o's Lemma,   we have 
\begin{equation} \label{gdp_bounded_process_check}
    P^{\bar \a^N}_{t,p+\eps_0}(\theta) - \varpi(\theta, X^{\bar \v^N}_{t,x}(\theta)) = P^N(\theta {\wedge \hat \tau^N}) - \varpi(\theta {\wedge \hat \tau^N}, X\txv(\theta {\wedge \hat \tau^N})) \geq 0 
\end{equation}
for any $[t,T]$-valued stopping time $\theta$, 
  which implies that $(\cdot,X^{\bar \v^N}_{t,x},P^{\bar \a^N}_{t,p+\eps_0}) \in \bar\Dc$ on $[t,T]$ a.s. Thus, the control $(\bar \v^N, \bar \a^N)\in \bar \Uc^N(t,x,p)$ for $\eps_0\ge 0$. \\

            % proof : convergence of value function V^N towards V
{\bf b.} We now prove the convergence result. We now fix $\eps_0>0$ and $\omega \in \Omega \backslash {\cal N}$ where ${\cal N}$ is the null set in \eqref{M_converge}. By (\ref{M_converge}),   there exists some $\hat{N}(\omega,\eps_0)$ such that for any $N \geq \hat{N}(\omega,\eps_0)$ 
$$
\sup_{[t, \tilde \tau(\omega)]} (\Ptpa-P^N)(\omega) \leq \frac{\eps_0}{2}{-\eps_0} \le 0.
$$
Since 
$
\brak{P^{\a}_{t,p} - \varpi(\cdot, X\txv)}(\omega) \geq 0,
$ on $[t, \tilde \tau(\omega) ]$, 
this implies that  $\tau^N(\omega) \geq \tilde \tau(\omega) $ for $N \geq \hat{N}(\omega,\eps_0)$, for $\P$-almost every $\omega\in \Omega$. Moreover, by Remark \ref{remark_aborbant_boundary} and the fact that ${\rm U}(t',x')$ is a singleton for all $(t',x')\in [0,T)\x \Rd$,  we have  $\v= \hat{\rm u}(\cdot, X\txv) $ $dt\times d\P$-a.e.~from $\tilde \tau(\omega)$ on. Hence,  for $N\ge \hat{N}(\omega,\eps_0)$, 
$ X^{\bar \v^N}_{t,x} (\omega) = X\txv  (\omega)\mbox{ on } [t,T]
$, for   $\P$-almost every $\omega \in \Omega$, by uniqueness of the solution to \eqref{def hat X dans Assumptio} as assumed in Assumption \ref{assum_selection_unique_opti}.

The above, combined with Remark \ref{rem : moments X} and the fact that $F$ and $f$ are continuous with polynomial growth, shows that 
$$
\liminf_{N\to \infty} \bar V^N(t,x,p+\eps_0)\ge \bar  V(t,x,p).
$$
On the other hand, by \eqref{prob_ori_sans borne} and \eqref{prob_ori_borne par N},
$$
 \bar V(t,x,p+\eps_0)\ge \limsup_{N\to \infty} \bar V^N(t,x,p+\eps_0).
$$
It remains to appeal to the continuity of $p\mapsto V(t,x,0,p)=\bar V(t,x,p)$ stated in Lemma \ref{lem: conti V en p} below to conclude.
\end{proof}

As mentioned above, the advantage of this approximation is that $\bar V^N$ solves a Hamilton-Jacobi-Bellman equation with bounded controls for which comparison holds in the class of semicontinuous functions with polynomial growth.
Define 
\begin{align*}
    &H^N(t,x,q,q_y,A):=\\ 
    &-\sup_{(u,a)\in U\x [-N,N]^d} \crl{ \hat \mu(x,u)^{\top} q + \frac12 {\rm Tr}\left[(\hat\sigma\hat\sigma^{\top})(x,u,a)A\right]+g(x,u)q_y+f(x,u,a)}
\end{align*}
for $(t,x,q,q_y,A)\in [0,T]\x \Rd\x \R^{ {2d}}  \x \R\x \S^{ {2d}}$. 

        % result : convergence of the auxiliary value function on the boundary Vb^N to the original value function on the boundary Vb
This is stated in terms of the upper- and lower-semicontinuous envelopes $\bar V^{N*}$ and $\bar V^N_*$ of $\bar V^N$, and with  $\bar {\cal V}^{N*}:=\bar V^{N*}(\cdot,\varpi)$ and $\bar {\cal V}^{N}_*:=\bar V^{N}_*(\cdot,\varpi)$.
\begin{Theorem}\label{thm : pde interior domain VN} Let the conditions of Proposition \ref{prop: w smooth}  and Assumption \ref{assum_selection_unique_opti} hold. Then, $\bar V^N_{*}$  is a viscosity supersolution  of 
\begin{align*} 
-\partial_{t}\vp + H^{N}(\cdot,D_{(x,p)}\vp,-D_p \vp,D^{2}_{(x,p)}\vp)&\ge 0 \mbox{ on $\inte_{P}\bar \Dc$},\\
\vp(T,\cdot)&\ge F \mbox{ on $\{(x,p)\in \R^{d+1}: p>G(x)\}$},
\end{align*}
and 
$\bar V^{N*}$ is a viscosity  subsolution of
\begin{align*} 
-\partial_{t}\vp + H^N(\cdot,D_{(x,p)}\vp,-D_p \vp,D^{2}_{(x,p)}\vp)&\le 0 \mbox{ on $\inte_{P}\bar \Dc\cup \partial_{<T} \bar \Dc$}\\
\vp(T,\cdot)&\le F \mbox{ on $\{(x,p)\in \R^{d+1}: p\ge G(x)\}$}.
\end{align*}
Moreover, $\bar \Vc^{N}_*$ and $\bar \Vc^{N*}$ are respectively viscosity super- and subsolutions of 
\eqref{eq: pde Vc}-\eqref{eq: cond bord Vc}.
\end{Theorem}

\begin{proof} {\bf a.}~It follows from step a.~in the proof of Proposition \ref{prop: approx N} that  $(t,x,p) \in \bar \Dc$ implies that one can find $(\nu,\alpha)\in \bar \Uc^N(t,x,p)$ such that $(\cdot,X^{ \v}_{t,x},P^{\a}_{t,p}) \in \bar\Dc$ on $[t,T]$. On the other hand, if $(\nu,\alpha)$ is such that $(\cdot,X^{ \v}_{t,x},P^{\a}_{t,p}) \in \bar\Dc$ on $[t,T]$  then  $(t,x,p) \in \bar \Dc$. Moreover, $\bar \Dc=\{(t,x,p)\in [0,T]\times \Rd\times \R: p\ge \varpi(t,x)\}$. Indeed, the same arguments  as in step a.~in the proof of Proposition \ref{prop: approx N} show that existence holds for the optimal control problem associated with $\varpi$ under the conditions of Proposition \ref{prop: w smooth}  and Assumption \ref{assum_selection_unique_opti}. Then, this follows from the same martingale representation argument as presented in  \cite[Proposition 3.1]{bouchard_stochastic_2010}. Hence, $(\nu,\alpha)\in \bar \Uc^N(t,x,p)$ if and only if $P^{\a}_{t,p}\ge \varpi(\cdot,X^{ \v}_{t,x})$ on $[t,T]$. Then, $V^N$ is reduced to a standard optimal control problem under state constraint and the derivation of the super- and subsolution properties are standard. See e.g.~the proof of \cite[Theorem 4.2]{bouchard_weak_2012}.\\

{\bf b.} For the super- and subsolution properties of $\bar \Vc^{N}_*$ and $\bar {\cal V}^{N*}$, it suffices to follows exactly the proof of Theorem \ref{thm : dirichlet condition} below. This is possible for $N\ge C_\varpi$ as defined in \eqref{def : varpi}. 
\end{proof}

We conclude this section with  the technical Lemma that was used in the proof of Proposition \ref{prop: approx N}.

        % lemma : value function V continuous in terms of constraint limit p
\begin{Lemma}\label{lem: conti V en p}   The map $p\in (w(t_{0},z_0),\infty)\mapsto V(t_{0},z_0,p)$ is continuous for all $(t_0,z_0)\in [0,T)\x \Rdd$.
\end{Lemma}

\begin{proof} Let $(t_{n},x_{n},y_n,p_{n})_{n\ge 1}$ be a sequence converging to $(t_{0},x_{0},y_0,p_{0})\in \inte \Dc$ such that 
$V(t_{n},z_n,p_{n})\to V_{*}(t_{0},z_0,p_{0})$ as $n\to \infty$, in which $z_{n}:=(x_{n},y_{n})$ and $z_{0}=(x_{0},y_{0})$. Let $\nu \in \Uc$ be such that $\E[{\rm G}^{\nu}_{t_{0},z_0}]\le p_{0}-\iota$ for some $\iota>0$. By Remark \ref{rem : moments X}, the continuity and polynomial growth of $G,g,F$ and $f$, we have  $\E[{\rm G}^{\nu}_{t_{n},z_n}]< p_{n} $ for $n$ large enough and $\lim_{n\to \infty}\E[{\rm F}^{\nu}_{t_{n},x_{n}}]=\E[{\rm F}^{\nu}_{t_{0},x_{0}}]$ up to passing to a subsequence. Therefore, 
$$
\liminf_{n\to \infty} V(t_{n},z_n,p_{n})\ge \E[{\rm F}^{\nu}_{t_{0},x_{0}}].
$$
By arbitrariness of $\nu$, this implies that 
$$
V_{*}(t_{0},z_0,p_{0})=\liminf_{n\to \infty} V(t_{n},z_n,p_{n})\ge V(t_{0},z_0,p_{0}-\iota)\ge V_{*}(t_{0},z_0,p_{0}-\iota).
$$
We now use Theorem \ref{thm : pde interior domain} to deduce that $V_{*}$ is a viscosity supersolution of $-D^{2}_{p} \vp\ge 0$ on  $\inte \Dc$, which, by the same arguments as in \cite[Proposition 5.2]{cvitanic1999closed}, implies that   $p\in (w(t_{0},z_0),\infty)\mapsto V_{*}(t_{0},z_0,p)$ is concave and therefore continuous.  { Using the above, we deduce that  
$$
V_{*}(t_{0},z_0,p_{0})=\lim_{\iota \downarrow 0}  V(t_{0},z_0,p_{0}-\iota).
$$
On the other hand, by the above and since $V$ is non-decreasing in its last argument, for all $\eps_o>0$,
$$
V_{*}(t_{0},z_0,p_{0})\le \lim_{\iota\downarrow 0} V(t_{0},z_0,p_{0}+\iota)\le \lim_{\iota \downarrow 0}  V(t_{0},z_0,p_{0}+\eps_o-\iota) =V_{*}(t_{0},z_0,p_{0}+\eps_o).
$$
Using the continuity of $p\in (w(t_{0},z_0),\infty)\mapsto V_{*}(t_{0},z_0,p)$ and again the fact $V$ is non-decreasing in its last argument, 
 we then deduce that } $V_{*}(t_{0},z_0,p_{0})= V(t_{0},z_0,p_{0})$. Hence $V(t,z,\cdot)=V_{*}(t,z,\cdot)$ is continuous on $(w(t,z),\infty)$, for all $(t,z)\in [0,T)\x \R^{d+1}$.
\end{proof}

% Approx. by smooth functions
\subsection{Approximation by smooth and uniformly elliptic coefficients} \label{sect_approximation cas non regulier}

In this section, we explain how the original problem \eqref{prob_ori_sans borne} can be approximated when it does not satisfy our Assumptions \ref{assum_g_holder}, \ref{assum_sigma_unif_elliptic} and \ref{assum_selection_unique_opti}, and in the case where $\Ur(t,x)$ is not reduced to a singleton. \\

    % Discussion on the assumptions
As for Assumption \ref{assum_g_holder}, it can be easily handled by considering a sequence of smooth approximations $(G_\eps)_{\eps>0}$ of $G$, while the uniform ellipticity condition of Assumption \ref{assum_sigma_unif_elliptic} is satisfied by adding a small regularizing term $\eps I_d$ to $\sigma$, in which $I_d$ denotes the $d\x d$-identity matrix. On the other hand, the Assumption \ref{assum_selection_unique_opti} might be more difficult to satisfy. For this let us consider an example. Assume that 
\begin{equation}\label{eq: u mapsto to example}
u\in U\mapsto \Lc_X^u \varpi(t,x)+g(x,u)
\end{equation}
is convex. Then, 
\begin{equation}\label{eq: u mapsto to example eps}
u\in U\mapsto \Lc_X^u \varpi_{\eps}(t,x)+g(x,u)+\eps |u|^2
\end{equation}
is strictly convex and therefore $\Ur_\eps(t,x)$ is reduced to a singleton, in which $\Ur_\eps$ and $\varpi_{\eps}$ are defined as $\Ur$ and $\varpi$ but with the map $g_{\eps}(x,u):=g(x,u)+\eps |u|^2$ in place of $g(x,u)$. Assume now that \eqref{eq: u mapsto to example} is linear, for all $(t,x)\in [0,T)\x \R^d$, and that $U=\Pi_{i=1}^d [\underline u^i,\overline u^i]$, then 
\begin{align*}
\hat \ur_\eps(t,x)&:={\rm argmin }\{\Lc_X^u \varpi_{\e}(t,x)+g(x,u)+\eps |u|^2: u\in U\}\\
&=\left(\underline u^i\vee \left(-\frac{1}{2\eps}\frac{\partial}{\partial u^i}\left( \Lc_X^u \varpi_{\e}(t,x)+g(x,u)\right)_{|u=u_0}\right)\wedge \overline u^i \right)_{1\le i\le d}
\end{align*}
in which the right-hand side does not depend on $u_0$, an arbitrary point in $U$. In this case, $\hat \ur_\eps$ is uniquely defined and H\"older continuous,  even  locally Lipschitz if $\sigma$ does not depend of $u$, see Proposition \ref{prop: w smooth}. Since $\mu$ and $\sigma$ are bounded, it turns out that it is enough to have $\hat \ur_\eps$ locally Lipschitz to construct a solution to the SDE 
\begin{equation*}
	\hat X_{t,x} (s) = x + \int_{t}^{s} \mu(\hat X_{t,x}(r),\hat \ur_\eps(r,\hat X_{t,x}(r))  dr + \int_{t}^{s} \sigma(\hat X_{t,x}(r),\hat \ur_\eps(r,\hat X_{t,x}(r)) dW_r
\end{equation*}
so that Assumption \ref{assum_selection_unique_opti} is satisfied. We refer to e.g.~\cite{krylov2005strong,veretennikov1981strong,zhang2010stochastic} and the references therein  for various results on the existence of solution of SDEs for non-Lipchitz coefficients.  

By Remark \ref{rem : unicite edp cal V}, having $\hat \ur_\eps$ locally Lipschitz on $[0,T)\x \Rd$ is also enough to ensure that comparison holds for  \eqref{eq: pde Vc}-\eqref{eq: cond bord Vc} within the class of functions with polynomial growth.
\vspace{2mm}
    % Result : convergence of approx. problem V_e to original problem V
    
Let us now formalize this. Let $(F_{\eps},G_{\eps},f_{\eps},g_{\eps},\mu_{\eps},\sigma_{\eps})_{\eps>0}$ be a family of maps on $[0,T]\x \R^{d}\x U$ with values in $\R^{4}\x \R^{d}\x \M^{d}$ satisfying the following assumption. 
        % assumption 
\begin{Assumption}\label{assum: suite regularisante} There exists $C,\ell>0$ and a family $(\rho_{N})_{N\ge 1}$ of real-valued maps on $\R^{d}$ such that 
\begin{itemize}
\item $(\mu_{\eps}(\cdot,u),\sigma_{\eps}(\cdot,u))_{\eps>0}$ is equi-lipschitz, uniformly in $u\in U$, and $(\mu_{\eps},\sigma_{\eps})_{\eps>0}$ is bounded.
\item $|(F_{\eps},G_{\eps},f_{\eps},g_{\eps})(x,u)|\le C(1+|x|^{\ell})$, for all $(x,u)\in \R^{d}\x U$ and $\eps>0$.
\item $\rho_{N}(0)=0$, $\rho_{N}$ is {locally Lipschitz-continuous}, and $|\rho_{N}(x)|\le C(1+|x|^{\ell_{N}})$ for all $x\in \R^{d}$, for some $\ell_{N}\ge 1$, for all $N\ge 1$.
\item $|F_{\eps}(x)-F_{\eps}(x')|+|G_{\eps}(x)-G_{\eps}(x')|+|f_{\eps}(x,u)-f_{\eps}(x',u)|+|g_{\eps}(x,u)-g_{\eps}(x',u)|\le \rho_{N}(x-x')$ for all $x,x'\in B_{N}$, $u\in U$, for all $N\ge 1$ and $\eps>0$.
\item $(F_{\eps},G_{\eps},f_{\eps}(\cdot,u),g_{\eps}(\cdot,u),\mu_{\eps}(\cdot,u),\sigma_{\eps}(\cdot,u))_{\eps>0}$ converges uniformly to $(F,G,f,g,\mu(\cdot,u),$ $\sigma(\cdot,u))$ on compact sets as $\eps\to 0$, uniformly in $u\in U$.
\end{itemize}
\end{Assumption}

For $\eps>0$, let $V_{\eps}$ be defined as $V$ but with respect to $(F_{\eps},G_{\eps},f_{\eps},g_{\eps},\mu_{\eps},\sigma_{\eps})$ in place of $(F,G,f,g,\mu,\sigma)$. 

        % statement of result
\begin{Proposition} Let Assumption \ref{assum: suite regularisante} hold. Then, $V_{\eps}\to V$ pointwise on $\inte \Dc$ as $\eps \to 0$.
\end{Proposition}

        % proof 
\begin{proof}   Recall that $(g_{\eps},f_{\eps},\mu_{\eps},\sigma_{\eps})_{\eps>0}$ converges uniformly on compact sets to  $(g,f,\mu,\sigma)$ and satisfies Assumption \ref{assum: suite regularisante}. 
Given $(t_{0},z_0,p_{0})\in \inte \Dc$, with $z_0:=(x_{0},y_0)$, let $X^{\eps}$ be defined as the solution of 
\begin{equation*}
	X^{\eps,\nu}(s) = x_{0} + \int_{t_{0}}^{s} \mu_{\eps}(X^{\eps,\nu}(r),\nu_{r}) dr + \int_{t_{0}}^{s} \sigma_{\eps}(X^{\eps,\nu}(r),\nu_{r})dW_r,\; s\le T.
\end{equation*}
Standard estimates imply that, for all $k\ge 1$, 
\begin{align}\label{eq: coonv unif Xeps}
\sup_{\nu \in \Uc}\E[\sup_{[t_{0},T]}|X^{\eps,\nu}-X^{\nu}_{t_{0},x_{0}}|^{k}]\to 0\mbox{ as } \eps\to 0.
 \end{align}
{\bf a.} Fix now $\nu \in \Uc$  such that $\E[{\rm G}^\nu_{t_{0},x_{0},y_0}]\le p_{0}-\iota$ for some $\iota>0$.   
By Assumption \ref{assum: suite regularisante}, given $N\ge 1$,  we have for some $C,\ell>0$, independent on $N$ and $\eps$,
\begin{align*}
&\E[|G(X^{\nu}_{t_{0},x_{0}}(T))-G_{\eps}(X^{\eps,\nu}(T))|]\\
&\le 
\E[|G(X^{\nu}_{t_{0},x_{0}}(T))-G_{\eps}(X^{\nu}_{t_{0},x_{0}}(T))|]\\
&+\E[|G_{\eps}(X^{\nu}_{t_{0},x_{0}}(T))-G_{\eps}(X^{\eps,\nu}(T))|]\\
&\le \sup_{B_{N}}|G-G_{\eps}|+\E[C(1+|X^{\nu}_{t_{0},x_{0}}(T)|^{\ell})\1_{\{|X^{\nu}_{t_{0},x_{0}}(T)|\ge N\}}]
\\
&+ \E[\rho_{N}(X^{\nu}_{t_{0},x_{0}}(T)-X^{\eps,\nu}(T))\1_{\{(X^{\nu}_{t_{0},x_{0}}(T),X^{\eps,\nu}(T))\in (B_{N})^{2}\}}]\\
& +\E[C(1+|X^{\nu}_{t_{0},x_{0}}(T)-X^{\eps,\nu}(T)|^{\ell_{N}})\1_{\{(X^{\nu}_{t_{0},x_{0}}(T),X^{\eps,\nu}(T))\notin (B_{N})^{2}\}}].
 \end{align*}
Up to changing $C$, \eqref{eq: coonv unif Xeps}, Assumption \ref{assum: suite regularisante} and the Markov inequality then imply that 
\begin{align*}
\limsup_{\eps\to 0}\E[|G(X^{\nu}_{t_{0},x_{0}}(T))-G_{\eps}(X^{\eps,\nu}(T))|]
&\le C\E[(1+|X^{\nu}_{t_{0},x_{0}}(T)|^{\ell})^{2}]^{\frac12}\frac{\E[|X^{\nu}_{t_{0},x_{0}}(T)|]^{\frac12}}{N^{\frac12}}
.
\end{align*}
Sending $N\to \infty$ thus implies that $\lim_{\eps\to 0}\E[|G(X^{\nu}_{t_{0},x_{0}}(T))-G_{\eps}(X^{\eps,\nu}(T))|]=0$. One can show similarly that 
\begin{align*}
0&=\lim_{\eps\to 0}\E[\int_{t_{0}}^{T}(|f(X^{\nu}_{t_{0},x_{0}}(s))-f_{\eps}(X^{\eps,\nu}(s))|+|g(X^{\nu}_{t_{0},x_{0}}(s))-g_{\eps}(X^{\eps,\nu}(s))|)ds]\\
	&+ \E[|F(X^{\nu}_{t_{0},x_{0}}(T))-F_{\eps}(X^{\eps,\nu}(T))|].
 \end{align*}
It thus follows that 
$$
\liminf_{\eps\to 0}V_{\eps}(t_{0},z_0,p_{0})\ge V(t_{0},z_0,p_{0}-\iota)
$$
so that, by continuity of $V(t_{0},z_0,\cdot)$ on $(w(t_{0},z_0),\infty)$, see Lemma \ref{lem: conti V en p}, and arbitrariness of $\iota>0$,  
$$
\liminf_{\eps\to 0}V_{\eps}(t_{0},z_0,p_{0})\ge V(t_{0},z_0,p_{0}). 
$$
{\bf b.} Conversely, by the same arguments as above, for $\iota>0$, 
$$
V(t_{0},z_0,p_{0}+\iota)\ge \limsup_{\eps\to 0}V_{\eps}(t_{0},z_0,p_{0})
$$
and therefore, 
$$
V(t_{0},z_0,p_{0})\ge \limsup_{\eps\to 0}V_{\eps}(t_{0},z_0,p_{0})
$$
by Lemma \ref{lem: conti V en p}. This concludes the proof.
\end{proof}

% ---------------------------------
%%% Numerical resolution - Deep Learning %%%
%----------------------------------

\section{Numerical resolution : a Deep Learning implementation} \label{sect_algo}

% Intro to the section - Organisation + precision on what is being trained
In this section, we introduce a 3-step deep-learning-based algorithm which sequentially estimates : (1) the optimal control on the domain boundary and the domain boundary $\varpi$ itself, (2) the value function $\bar \Vc$ on the domain boundary, and (3) the value function $\bar V$ on the entirety of the domain along with the optimal control for the state process and the Martingale representation. Without loss of generality, given Remark \ref{rem : V without y}, we can fix $y=0$. The PDE characterizations obtained in the previous sections are used to train and validate the quality of our estimation.\\

% Discussion on the discretization
We first partition the time horizon into a $K$-period  discrete time grid $\Pi_K:=\{t_k = \frac{kT}{K},\; k=0,...,K\}$ and use an Euler scheme to approximate the  state process $X^\v$, meaning that
\begin{align} 
    X^\nu_{t_0} &= X_{t_0} \nonumber \\
    X^\nu\tkk &= X^\v\tk + \mu(X^\nu\tk , \v\tk) \Delta t + \sigma(X^\v\tk , \v\tk)\Delta W\tkk \label{dyn_X_discrete}
\end{align} 
where $X_{t_0}$ is an initial state valued in $\R^d$ that will be drawn from a given distribution, $\v$ is an admissible control, $\Delta t := t_{k+1} - t_k = \frac{T}{K}$, and $\Delta  W\tkk := W\tkk - W\tk$. 
\vspace{2mm}

We denote by $\Theta$ a set of feasible parameters for the neural network structures used below. 

% Estimation of u and w
\subsection{Estimation of the domain boundary and the optimal control along the boundary} \label{subsect_train u and w}

    % Train u : architecture + loss function 
We first estimate the optimal control along the boundary which verifies \eqref{eq: HJB w}. Given the Markovian setting of the problem, we can naturally assume that any admissible control can take the form of a Borel map of time and state process. We hence restrict our search to controls of the form
\begin{equation*}
    \v^{\t, \varpi}\tk := n^{u,\varpi}_\t(t_k, X^{\v^{\t, \varpi}}\tk)
\end{equation*}
in which $n^{u,\varpi}_\t$ is a  fully connected neural network parameterized by $\t \in \Theta$ mapping from $\Pi_K \x \Rd$ to $U$.\\

To train $n^{u,\varpi}_\t$, we sample $J$ independent initial values $(X^j_{t_0})^{j=1, ..., J}$ along with $J$ independent paths $(\Delta W\jtk)_{k=1,...K}^{j=1,...,J}$. Then, recursively, we compute the corresponding control at each $t_k \in \{t_0, ..., t_{K-1}\}$ and roll forward the state variable following the Euler scheme \eqref{dyn_X_discrete}, i.e
\begin{align*}
    \v\tjw_{t_k} &= n^{u,\varpi}_\t(t_k, X^{\v\tjw, j}\tk) \\
    X^{\v\tjw, j}\tkk &= X^{\v\tjw, j}\tk + \mu(X^{\v\tjw, j}\tk, \v\tjw_{t_k}) \Delta t + \sigma(X^{\v\tjw, j}\tk, \v\tjw_{t_k}) \Delta W\jtkk 
\end{align*}
for $k = 0, ...,K-1$ and $j = 1,...,J$. Consequentially, to find the optimal control network (given the above-mentioned architecture), we define a training loss function which is the Monte-Carlo counterpart of \eqref{eq: def w} over the parameter $\t \in \Theta$: 
\begin{equation} \label{loss funcu_NN}
    \mathbf L^{u,\varpi}(\theta) := \frac1J \sum^{J}_{j=1} \left[ G( X^{\nu\tjw,j}_T) + \frac1K \sum^{K-1}_{k=0} g(  X^{\v\tjw,j}\tk, \v\tjw\tk) \right].
\end{equation}
We denote by $\hat\t^{u}_\varpi$ the optimal parameter obtained from \eqref{loss funcu_NN}, by $\hat \v^\varpi := n_{\hat \t^u_\varpi}^{u,\varpi}$, and by 
\begin{equation*}
    \hat \nu^{\varpi,j}:=\hat \nu^{\varpi}(\cdot,X^{\hat \nu^\varpi,j})
\end{equation*}
the control associated with the $j$-th path. \\

    % Train w : architecture + loss function + error measure for u and w
Now we estimate  $\varpi$ by using its PDE characterization as given in \eqref{eq: HJB w} with Physics-Informed Neural Network (PINN) method as first introduced in \cite{raissi_2019} \footnote{See Appendix \ref{append : discussion on implementation} for further details on the implementation of the PINN method.}. Namely, given another neural network $n^{\varpi}_\theta$, we  seek for $\hat \theta_\varpi$ which minimizes over $\theta\in \Theta$ the loss function 
\begin{align} 
    \mathbf L^\varpi(\theta) := \frac1J & \sum^J_{j=1} \frac1K \sum^{K-1}_{k=0} \left| \Lc^{\vhat^{\varpi,j}\tk}_X n^\varpi_\t (t_k, X^{\vhat^{\varpi,j},j}\tk)+g(X^{\vhat^{\varpi,j},j}\tk, \hat\v^{\varpi, j}\tk)\right|^2  \nonumber \\
    & +\frac1J \sum^J_{j=1}   \left| n^\varpi_\t(T, X^{\vhat^{\varpi,j},j}_T) - G(X^{\vhat^{\varpi,j},j }_T)\right|^2. \label{loss funcw_nn}
\end{align}

%%%%%%%%%%
This optimization leads to an estimation $\hat \varpi:=n^{\varpi}_{\hat \theta_\varpi}$ of $\varpi$. To verify the precision of our estimation, we test the result by checking how it deviates from the solution of  \eqref{eq: HJB w} with boundary condition $G$ at $T$. 
For this, we use the following heuristics:
\begin{align*}
\inf_{u\in U} \Lc^u_X \hat \varpi +g
&= \eps^u_\varpi+\eps^\varpi_\varpi\;\mbox{on } [0,T)\x \R^d,\;
\hat \varpi(T,\cdot)=G+\eps^T_\varpi\mbox{ on $\Rd$,}
\end{align*}
where 
\begin{align*}
\eps^u_\varpi(t,x)&:=\inf_{u\in U} \Lc^u_X \hat \varpi(t,x)-\Lc^{\hat \nu^{\varpi}(t,x)}_X \hat \varpi(t,x)
\\
\eps^\varpi_\varpi(t,x)&:=\Lc^{\hat \nu^{\varpi}(t,x)}_X \hat \varpi(t,x)+g(x,\hat \nu^{\varpi}(t,x))\\
\eps^T_\varpi(x)&:=\hat \varpi(T,x)-G(x)
\end{align*}
for $(t,x)\in [0,T]\x \Rd$. Hence, 
$$
\hat \varpi(t,x)=\inf_{\nu \in \Uc}\E\left[G(X_{t,x}^\nu(T))+\int_t^T g(X_{t,x}^\nu(s),\nu_s)ds+\eps^T_\varpi(X_{t,x}^\nu(T)) +\int_t^T\left[\eps^u_\varpi+\eps^\varpi_\varpi\right](s,X^\nu_{t,x}(s)ds\right],
$$
so that
\begin{align*}
    |\hat \varpi(t,x)- \varpi(t,x)|
    \le \sup_{\nu \in \Uc}\left|\E\left[\eps^T_\varpi(X_{t,x}^\nu(T))+\int_t^T \left[\eps^u_\varpi+\eps^\varpi_\varpi\right](s,X^\nu_{t,x}(s)ds\right]\right|.
\end{align*}
We can thus use the discrete time counterpart of the above  to estimate the error. To simplify, we do not perform the optimization over $\nu$ and simply estimate the normalized error: 
\begin{equation}\label{error measure Vb}
    \Ec^\varpi := \frac{\delta^\varpi}{\frac1J \sum^J_{j=1} \bigl|\hat\varpi(t_0, X^{\hat\v^{\varpi, j}, j}_{t_0})\bigr|}
\end{equation}
where 
\begin{equation} \label{numerator of error measure u and w}
\delta^{\varpi}:=\frac1J \sum_{j=1}^J\left\{ |\eps^T_\varpi(\check X^{\hat \nu^{\varpi,j},j}_{T})|+\frac1K\sum_{k=0}^{K-1} \Big(|\eps^u_\varpi|+|\eps^\varpi_\varpi|\Big)(t_k, \check X^{\hat \nu^{\varpi,j},j}\tk)\right\}.
\end{equation}

%%%%%%%%%%

% Estimation of Vb 
\subsection{Estimation of the value function on the domain boundary}
\label{subsect_train Vb}
    % Train Vb : architecture + loss function + error measure
Similar to the estimate of $\varpi$ in the previous section, we estimate $\bar \Vc$ by using its PDE characterization in Theorem \ref{thm : dirichlet condition} with a chosen neural network architect $n^{\bar \Vc}_\t$. We assume here that there is only one feasable control, which is the one estimated in the previous step. More concretely, we use PINN method to find the minimizer $\hat \t_{\bar \Vc}$ among feasible parameters $\Theta$ for the loss function 
\begin{align}
    \label{loss funcV_bound}
    \mathbf L^{\bar \Vc}(\t) := \frac1J & \sum_{j = 1}^J \frac1K \sum^{K-1}_{k=0} \left| \Lc^{\vhat^{\varpi,j}\tk}_X n^{\bar \Vc}_\t (t_k, X^{\vhat^{\varpi,j}, j}_{t_k}) +f(X^{\vhat^{\varpi,j}, j}\tk, \vhat^{\varpi, j}\tk) \right|^2 \nonumber \\
     & + \frac1J \sum_{j = 1}^J \left|n^{\bar \Vc}_\t (T, X^{\vhat^{\varpi, j}, j}_T) - F(X^{\vhat^{\varpi,j}, j}_T) \right|^2
\end{align}
where $X^{\vhat^{\varpi,j}, j}_{\cdot}$ is the $j$-th path generated by $\vhat^{\varpi,j}$ as described above in Section \ref{subsect_train u and w} \footnote{Note that even though this training can use the same training sample as in Section \ref{subsect_train u and w}, we generate a new sample to enhance the generalization of our training.}. We set $\hat \Vc := n^{\bar \Vc}_{\hat \t_{\bar \Vc}}$ as the estimate for $\bar\Vc$ with $\hat \t_{\bar \Vc}$ being the optimizer of \eqref{loss funcV_bound}. Similar to the previous section, we also propose an error measure for the joint training of $\hat \v^\varpi$ and $\hat \Vc$ defined by 
\begin{equation}\label{error measure Vb}
    \Ec^\Vc := \frac{\delta^\Vc}{\frac1J \sum^J_{j=1} \bigl|\hat \Vc(t_0, X^{\hat\v^{\varpi, j}, j}_{t_0})\bigr|}
\end{equation}
where 
\begin{equation}
    \delta^{\Vc} := \frac1J \sum^J_{j=1} \left\{ \left|\eVbT( X^{\vhat^{\varpi,j},j}_T)\right| + \frac1K\sum^{K-1}_{k=0} \Big(|\eps^u_\Vc|+|\eps^\Vc_\Vc|\Big)(t_k, X^{\vhat ^{\varpi,j}, j}\tk) \right\} 
\end{equation}
where 
\begin{align*}
    \e^u_\Vc &:= \max_{{\rm U}(t,x)}\Lc^u_X \hat \Vc (t,x) -  \Lc^{\vhat^\varpi(t,x)}_X \hat \Vc (t,x)\\ 
    \eVbVb(t,x)&:= \Lc^{\vhat^\varpi(t,x)}_X \hat \Vc (t,x) + f(x,\vhat^\varpi(t,x)) \\
    \eVbT(x) &:= \hat \Vc(T, x) - F(x)
\end{align*}
for $(t,x) \in [0,T] \times \Rd$.

% Estimation of V
\subsection{Estimation of the value function and the associated optimal control} \label{subsect_train V}

    % Train a : architecture + loss function
For the interior of the viable domain, we have to take into account the control $\a$ for the martingale representation $\Ptpa$. Consequentially, we consider controls of the form 
$$\upsilon\tV\tk = (\v\tV\tk, \a\tV\tk) := n^{u, a, V}_\t(t_k, X^{\v\tV}\tk, P^{\a\tV}\tk)$$
where $n^{u,a,V}_\t$ is a $U \x [-\Nc, \Nc]$-valued neural network parameterized by some $\t \in \Theta$. We independently draw a sample of initial states $(X^j_{t_0})^{j=1,..J}$ and brownian paths $(\Delta W^j_{t_k})^{j=1,..J}_{k=1,...K}$ from the corresponding distributions, and as for the initial constraints, we generate a new sample $(P^j_{t_0})^{j=1,..J}$ with
\begin{equation} \label{sampling P}
    P^j_{t_0} = \hat\varpi(t_0, X^j_{t_0}) + \epsilon_p^j
\end{equation}
where $\epsilon_p^j$ is drawn uniformly from $[0, \Ec_P]$ for some  $\Ec_P > 0$ significantly large. Then, the state process can be diffused in a similar manner as in Section \ref{subsect_train u and w}, meaning that, for each $j$-th path, we follow the scheme  
\begin{align*}
    \upsilon\tjV\tk & = (\v\tjV\tk, \a\tjV\tk) = n^{u, a, V}_\t(t_k, X^{\v\tjV, j}\tk, P^{\a\tjV, j}\tk)\\
     X^{\v\tjV, j}\tkk &= X ^{\v\tjV, j}\tk + \mu(X^{\v\tjV, j}\tk, \v\tjV\tk) \Delta t + \sigma(X^{\v\tjV, j}\tk, \v\tjV\tk)\Delta W^j\tkk \\
     P^{\a\tjV, j}\tkk & = P^{\a\tjV, j}\tk +  (\a\tjV\tk )^\top\Delta W^j\tkk
\end{align*}
for $ k = 0, ..., K - 1$. In light of the arguments in Section \ref{sect_approximation control bornes}, we restrict the set of values for $\a$ to $[-\Nc, \Nc]$ for some $\Nc > 0$ sufficiently large. \\

Then, among the possible $\t \in \Theta$, we look for the optimal $\hat \t^\upsilon_V$ which minimizes the Monte-Carlo counterpart of \eqref{prob_ori_sans borne}. To take into account  the domain constraint,  we impose an additional penalty term in the loss function to penalize any time the martingale  $P^{\a^{\t,V}}$ goes below the domain boundary $\hat \varpi(\cdot,X^{\v\tjV})$. Concretely, the loss function for our training is \\
\begin{align} \label{loss function u_a_nn}
     \mathbf{L}^{u,a,V}(\t) := - \frac1J &\sum^J_{j=1} \left[F(X^{\v\tjV,j}_T) + \frac1K \sum^{K-1}_{k=0}f( X^{\v\tjV,j}\tk, \v\tjV\tk) \right] \nonumber\\
        & + \frac1J \sum^J_{j=1} \frac1K \sum^{K-1}_{k=0} \left( P^{\a\tjV, j}\tk - \hat\varpi(t_k, X^{\v\tjV,j}\tk)\right)^2.
\end{align} 
Let us denote by  $\hat \upsilon^V := n^{u,a,V}_{\hat \t^\upsilon_V}$ the estimated optimal control function and by 
$$
(\hat\v^{V,j}, \hat \a^{V,j}) := \hat \upsilon^{V,j} := \hat \upsilon^V(\cdot, X^{\hat\v^{V, j},j}_\cdot, P^{\hat\a^{V, j},j}_\cdot)
$$
the optimal control associated with the $j$-th path. \\

    % Train V : 
We then proceed with the estimation of the value function $\bar V$ using its PDE characterization in Theorem \ref{thm : pde interior domain} while applying the PINN method similar to Section \ref{subsect_train Vb}. In particular, we consider neural networks of the form $n^V_\t$ for some $\t \in \Theta$ and look for the minimizer $\hat \t_V$ of the loss function 
\begin{align} \label{loss func V}
    \mathbf L^V(\t) :=\frac1J &\sum^J_{j=1} \frac1K \sum^{K-1}_{k=0}\left| \Hc^{\hat\upsilon^{V,j}\tk} n^V_\t(t_k, X^{\hat\v^{V,j},j}\tk,P^{\hat\a^{V,j},j}\tk+\epsilon^j_k) + f(X^{\hat\v^{V,j},j}\tk, \hat \v^{V,j}\tk)\right|^2 \nonumber \\ 
    & + \frac1J\sum^J_{j=1} \frac1K \sum^{K-1}_{k=0}\left|n^V_\t\left(\cdot, \hat \varpi(\cdot)\right) -\hat \Vc\right|^2(t_k, X^{\hat\v^{V,j},j}\tk) \nonumber \\
    & + \frac1J\sum^J_{j=1}\left|n^V_\t(T,X^{\hat \v^{V,j},j}_T,P^{\hat \a^{V,j},j}_T) - F(X^{\hat \v^{V,j},j}_T)\right|^2 
\end{align}
where $\epsilon^j_k := 2\left(\hat\varpi(t_k, X^{\hat\v^{V,j},j}\tk)-P^{\hat\a^{V,j},j}\tk\right)^+$ and 
\begin{equation*}
    \Hc^v \varphi := \partial_t \varphi + \hat\mu(\cdot,u)^\top D_{(x,p)}\varphi + \tr{\hat\sigma\hat\sigma^\top(\cdot, u,a)D^2_{(x,p)}\varphi} - g(\cdot, u) \partial_p \varphi + f(\cdot, u)
\end{equation*}
for any mapping $\varphi \in C^{1,2}_b([0,T]\x\Rd\x\R, \R)$ and any control $v =(u,a) \in U \x [-\Nc, \Nc]$. The first line in \eqref{loss func V} corresponds to the PDE in the interior of the domain (the $\epsilon^j_k$s ensuring that the evaluated points actually belong  to the interior of the domain - up to the possibility of having $\hat\varpi(t_k, X^{\hat\v^{V,j},j}\tk)=P^{\hat\a^{V,j},j}\tk$, which we neglect here), the second and third  lines account for the Dirichlet conditions on the boundary. 
\\

    % Error measure : 
We   denote by $\hat V := n^V_{\hat \t_V}$ the neural network with the optimal parameters $\hat\t_V$, and, as in the previous steps,  we introduce the  error measure  
\begin{equation}\label{error measure V}
    \Ec^V := \frac{\delta^{\bar V}}{\frac1J \sum^J_{j=1}\bigl|\hat V(t_0, \check X^{\hat\v^{V,j}, j}_{t_0},\check P^{\hat\a^{V,j}, j}_{t_0})\bigr|}
\end{equation}
where  
\begin{equation}
    \label{numerator error measure V}
    \delta^{\bar V} := \frac1J\sum^J_{j=1}\left[\frac1K\sum^{K-1}_{i=0}\Big( \big| \e^\upsilon_V\big| + \big| \e^V_V\big| + \big| \e^b_V\big|\Big)(t_k, \check X^{\hat\v^{V,j},j}\tk,\check P^{\hat\a^{V,j},j}\tk)+\left|\e^T_V(\check X^{\hat\v^{V,j},j}_T, \check P^{\hat\a^{V,j},j}_T)\right|\right]
\end{equation}
with 
\begin{align*}
    \e^\upsilon_V \txp &:=\mathbf{1}_{\{p > \hat \varpi(t,x)\}} \left(\max_{v \in U \times [-\Nc,\Nc]} \Hc^v\hat V - \Hc^{\hat \upsilon^V}\hat V\right)\txp \\
    \e^V_V (t,x,p) &:= \mathbf{1}_{\{p > \hat \varpi(t,x) \}} \Hc^{\hat \upsilon^V}\hat V\txp \\
    \e^b_V(t,x,p) &:= \mathbf{1}_{\{p = \hat \varpi(t,x) \}} \left(\hat V\txp - \hat\Vc(t,x)\right) \\
    \e^T_V(x,p) &:= \hat V\txp - F(x) 
\end{align*}
for any $\txp \in [0,T]\times \Rd \times \R$ and  $v = (u,a) \in U \x [-\Nc,\Nc]$. 

% Alternative training method
\subsection{Alternative estimation method}

    % Different architecture 
Besides the estimation method proposed above, we have also experimented with a different architecture inspired by \cite{han_deep_2016} for the control neural networks $\v^{\t, \varpi}$ and $\upsilon^{\t, V}$. More specifically, we  considered a form of ``nested" neural network consisting of a sequence of sub-networks where each sub-network is a fully connected neural network designated to estimate uniquely the control at time $t_k$. We  also experimented with changing the diffusion process to reflect the absorbing nature of the domain boundary by imposing to follow the space boundary once it is reached (or crossed). Overall, results were pretty poor and the optimal control trajectories we obtained with this approach were quite counterintuitive. Still we explained below this alternative approach.\\

    % Optimal control on the boundary
        % architecture
Firstly, for the optimal control on the boundary $\v^{\t, \varpi}$, we tried to  take a neural network of the form
$$
\tilde n^{u, \varpi}_\t = (\tilde n^{u, \varpi, k}_\t)_{k = 0,..., K-1}
$$
such that for each $t_k \in \Pi_K$, the sub-network $\tilde n^{u, \varpi, k}_\t$ is a fully-connected neural network with $l$ hidden layers $h^1_k, ...h^l_k$, as shown in Figure \ref{fig: Nested NN}. This implies that the discrete time partition $\Pi_K$ is practically embedded in the nested structure of the global neural network, and the number of parameters of the global network $\tilde n^{u, \varpi}_\t$ scales linearly with the number of periods $K$.

        % diffusion scheme
\begin{figure}[H]
    \centering
    \includegraphics[width=\textwidth]{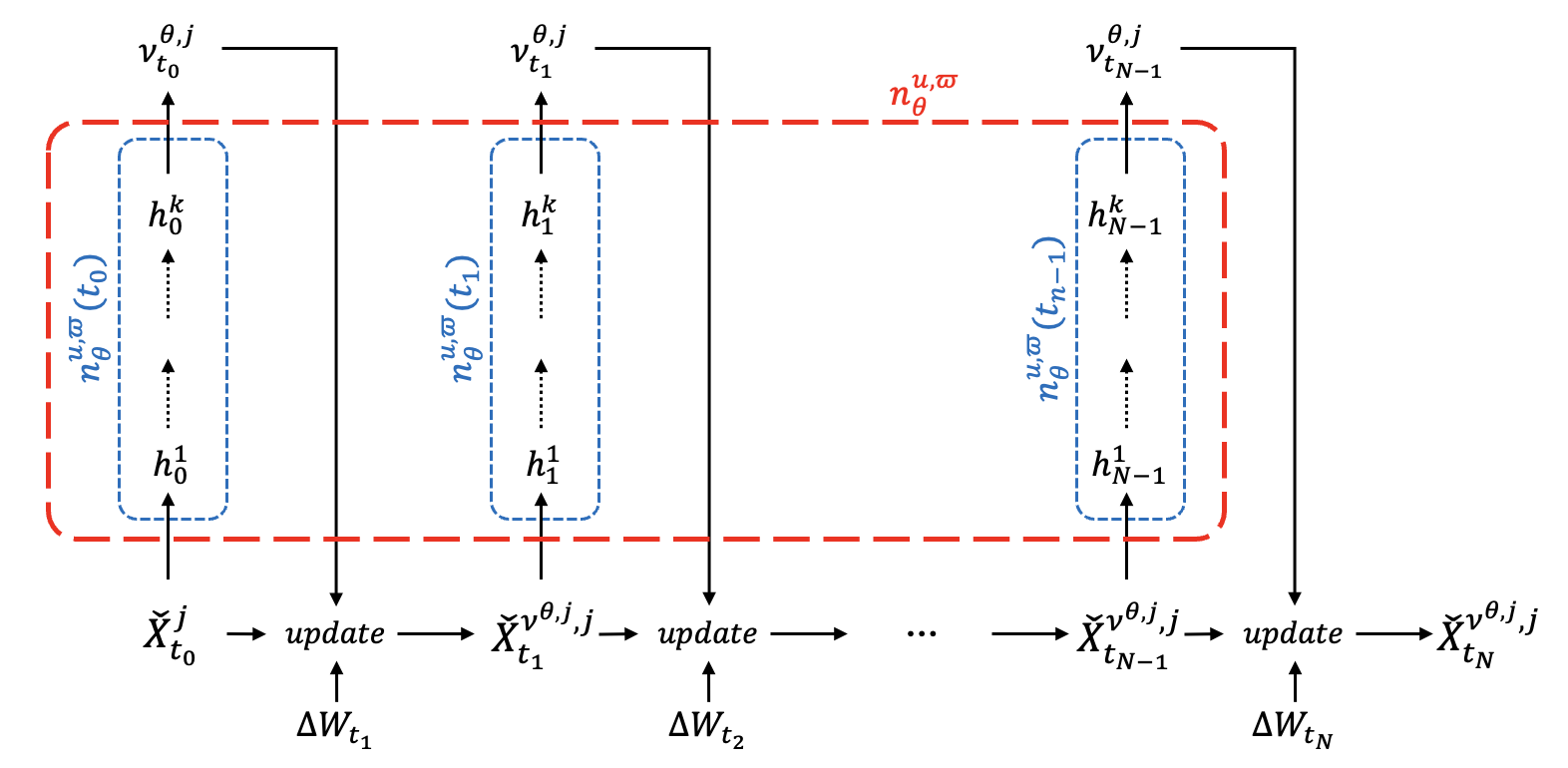}
    \caption{Diffusion of the state process $X^{\tilde \v\tjw,j}$ with neural network $\tilde n^{u, \varpi}_\t$} with the alternative ``nested" architecture.
    \label{fig: Nested NN}
\end{figure}
The diffusion of the state process $X^{\tilde \v ^{u, \varpi}}$ then follows the Euler scheme: \footnote{We use the notation $\tilde \v^{u, \varpi}$ to distinguish the control process by alternative architecture from the one by the officially chosen method, $\v^{u, \varpi}$, in Section \ref{subsect_train u and w}} 
\begin{align*}
    \tilde \v\tjw\tk &:= \tilde n^{u, \varpi, k}_\t(X^{\tilde \v\tjw,j}\tk) \\
    X^{\tilde \v\tjw,j}\tkk &= X^{\tilde \v\tjw,j}\tk + \mu(X^{\tilde \v\tjw,j}\tk, \tilde \v\tjw\tk) \Delta t + \sigma(X^{\tilde \v\tjw,j}\tk, \tilde \v\tjw\tk)\Delta W^j\tkk,
\end{align*}
given some initial states $(X^j_{t_0})^{j=1, ... J}$ and brownian increments $(\Delta W^j\tk)^{j=1, ... J}_{k = 1,...K}$. The training proceeds with the  minimization of  the same loss function as in \eqref{loss funcu_NN}.\\

    % Optimal control in the interior 
As for the optimal control in the interior of the domain, in addition to the structure with sub-networks introduced above, we also experimented with changing the diffusion 
process to reflect the absorbing nature of the domain boundary. In particular, for a neural network of the form
$$
\tilde n^{u, a, V}_\t = (\tilde n^{u,a,V,k}_\t)_{k=0,...K-1},
$$
where the sub-network $\tilde n^{u,a,V,k}_\t$ is $U\x [-\Nc, \Nc]$-valued, we denote by
\begin{equation*}
    \tilde \upsilon\tjV\tk = (\tilde \v\tjV\tk, \tilde \a\tjV\tk) := \tilde n^{u,a,V,k}_\t(X^{\tilde \v\tjV,j}\tk, P^{\tilde \a\tjV,j}\tk)
\end{equation*}
the control estimated by the $k$-th subnetwork for the $j$-th path at $t_k$ and by
\begin{equation*}
    \tilde \tau^j := \inf\{ t_k \in \Pi_K : \hat \varpi(t_k, X^{}\tk) \geq P^{\tilde \a}\tk\} \wedge T
\end{equation*}
the first time the $j$-th state process reaches (or crosses) the domain boundary. To enforce the absorbing nature of the domain boundary, we explicitly force the state control estimated by $\tilde n^{u,a,V}_\t$ to equate the state control on the boundary as estimated by $\tilde n^{u, \varpi}_\t$, meaning that
\begin{align*}
   \tilde \upsilon^{\t, V,j}\tk &= 
        \begin{cases}
            (\tilde \v^{\t,V,j}\tk, \tilde \a^{\t,V,j}\tk) & \text{ if } t_k \leq \tilde \tau^j \leq T \\
            (\tilde \v^{\t, \varpi,j}\tk, 0) &\text{ if } \tilde \tau^j < t_k \leq T ,
        \end{cases}
\end{align*}
and the Martingale $P^{\tilde \a\tjV, j}$ is set to be equal to $\hat \varpi(\cdot,X^{\tilde \v\tjV, j})$ in the case of reaching the spatial boundary before the end of the time horizon, while the state process follows the Euler scheme. Meaning that:
\begin{align}
    X^{\tilde \v\tjV, j}\tkk &= X^{\tilde \v\tjV, j}\tk + \mu(X^{\tilde \v\tjV, j}\tk, \tilde \v\tjV\tk) \Delta t + \sigma(X^{\tilde \v\tjV, j}\tk, \tilde \v\tjV\tk)\Delta W^j\tkk \nonumber \\
    P^{\tilde \a\tjV, j}\tkk &= 
        \begin{cases}
            P^{\tilde \a\tjV, j}\tk + (\tilde \a\tjV\tk)^\top \Delta W^j\tk & \text{ if } t_k \leq \tilde \tau^j \leq T \\
            \hat \varpi(t_k, X^{\tilde \v\tjV, j}\tk) & \text{ if } \tilde \tau^j < t_k \leq T 
        \end{cases} \label{dyn P discrete}
\end{align}
for any $k = 0, ..., K - 1$. With such diffusion, it is important to keep track of the stopping time $\tilde \tau$, but there is no further need for the penalty term in the training loss function for $\tilde n^{u,a,V}_\t$. Therefore, the alternative for the loss function \eqref{loss function u_a_nn} was 

\begin{equation*}
    \mathbf{\tilde L}^{u,a,V}(\t) := -\frac1J \sum^J_{j=1} \left[F(X^{\tilde\v\tjV,j}_T) + \frac1K \sum^{K-1}_{k=0}f( X^{\tilde\v\tjV,j}\tk, \tilde\v\tjV\tk) \right].
\end{equation*}

    % Comparison in terms of training results
All other estimation steps, which involve the PINN training guided by the PDE characterization, remained unchanged. 
%The \footnote{Python implementation code for the alternative method can be found at \url{https://github.com/KAPhm/optimal_control_alternative_method}.} \\

In comparison to the  method introduced in Sections \ref{subsect_train u and w}-\ref{subsect_train Vb}-\ref{subsect_train V}, this alternative method yields poorer performance when applied to the practical model introduced in Section \ref{subsect_ALM model} and proved to have more  difficulties in finding the optimal parameters. This is potentially due to the complexity of the nested architecture which scales linearly with the number of periods $K$. Another possible reason for this gap in performance can also be the treatment of the time variable. In particular, time is taken as an input for the neural network in the  algorithm of Sections \ref{subsect_train u and w}-\ref{subsect_train Vb}-\ref{subsect_train V}, and enhancement techniques such as encoding time variable with a sinusoidal encoder, as inspired by \cite{vaswani_attention_2023}[Section 3.5], can be implemented. On the contrary, this possibility of adding such enhancement in the alternative method is nullified due to the fact that the time horizon is already implicitly embedded in the ``nested'' architecture of the network. \\

On the other hand, with the explicit embedding of the domain boundary's absorption described in \eqref{dyn P discrete},  there is theoretically lower chance of having an ``overshoot'' due to discretization, i.e $P^{\tilde \a\tjV,j} < \hat \varpi(\cdot, X^{\tilde \v\tjV,j})$. More concretely, the diffusion in \eqref{dyn P discrete} guarantees that overshoot is possible solely at $\tilde \tau^j$ in the alternative method whereas overshoot due to discretization can happen at any moment and can increase over time  in the method of Sections \ref{subsect_train u and w}-\ref{subsect_train Vb}-\ref{subsect_train V}, and is only minimized through the penalty term introduced in the loss function \eqref{loss function u_a_nn}. On the other hand, the presence of the stopping time in \eqref{dyn P discrete} seemed to create difficulties for the optimization process of the neural network.

% ---------------------------------
%%% Application to practical model %%%
%----------------------------------
\section{An application to Asset-Liability Management} \label{sect_toy example}

This work is initially motivated by applications to an Asset-Liability Management problem. In this section, we describe a very rustic toy model as a first example of practical implementation of our theoretical results. 

% Presentation of ALM problem 
\subsection{An asset-liability management problem from life insurance} \label{subsect_ALM model}
Let us consider a toy model inspired by the Asset-Liability Management (ALM) problem for a profit-sharing life insurance product in France, commonly known as \textit{Fond Euro}. \\

We assume a market with one risky asset $S$ without dividend and a fixed risk-free interest rate $r$. At any moment $t$ within the time horizon of interest $[0,T]$, the status of the portfolio is represented by the market value of the asset $S_t$, the quantity of asset $\phi_t$, the amount of cash $\b_t$, and the insurer's liability towards the insured $L_t$, which is also known as \textit{Mathematical Provision}. Following the accounting convention, we consider the book value $\tilde S_t$ of the risky asset. For simplification, we assume that the risky asset is available for trading without transaction cost at any moment $t$ in the time horizon \footnote{This implies that there is neither default risk nor illiquidity risk in this market.} and that the book value is constant throughout the time horizon, \textit{i.e.} $\tilde S_t = \tilde S_0 =: \tilde S$ for all $0\leq t \leq T$. Hence, the state variable in this model is $X = (S, \phi, \b, L)$.  The strategy of the insurer is represented by the control variable $$\nu = (\dot \phi, \pi)$$ 
where $\dot \phi$ is the trading intensity of the risky asset and $\pi$ is the proportion of the financial profit to be shared with the insured. It should be noted that the insurer is a price taker whose (trading) decision does not affect the asset market value $S$. Furthermore, the insurer's control is limited by the market and the regulation, meaning that  $\dot \phi$ must stay within a fixed range $[\underline{ \dot \phi}, \overline {\dot \phi}]$ (market trading capacity) and likewise, $\pi$ within $[\underline \pi, \overline \pi]$ (regulation requirement on profit sharing).

% Mathematic framework 
\subsection{Mathematical framework for life insurance ALM}
    % Asset + Liability modeling (dynamic under control)
We assume that the price of the risky asset follows a Black-Scholes dynamic 
\begin{equation*}
    dS_t = \mu_S S_t dt + \sigma_S S_t dW_t 
\end{equation*}
where $W$ is a Brownian motion, $\mu_S \in \R$, and $\sigma_S$ is a positive constant. Under the French accounting convention, the profit producted at $t$ includes the inflows generated by all assets, which is the interest from cash in this case, and the capital gain/loss realized in case of an asset sale. Hence, the intensity of profit produced $\dot \eta_t$ at $t$ is   
\begin{equation} \label{eq: def profit intensity}
\dot \eta^\v_t = \beta^\v_t r + (\dot \phi_t)^-  (S_t - \tilde S)
\end{equation}
where  $(\dot \phi_t)^- = \max (0, -\dot \phi_t)$ is the quantity of the risky asset sold at $t$. Out of the total profit generated, a proportion $\pi_t$ is shared with the insured, and after sharing the financial profit, the insurer observes a certain level of lapses which can be decomposed into 2 types, cyclical and dynamic surrenders. More specifically, dynamic surrender is triggered by the difference between the profit shared by insurer and the risk-free interest that the insured could have received from a simple saving account. On the other hand, cyclical surrender is not related to the amount of profit shared but rather caused by non-financial reasons such as mortality, disability, or lifestyle's cycle. We define the lapse intensity as
\begin{equation} \label{eq: def lapse intensity}
\dot \gamma^\v_t := \left[\gamma^c + \gamma^d\left(r - \frac{\pi_t \dot \eta^\v_t}{L^\v_t}\right)\right]{\bf 1}_{\{L^\v_t>0\}}
\end{equation}
where $\gamma^c$ represents the cyclical lapse intensity and $\gamma^d $ is the scaling factor showing the (linear) relationship between dynamic surrender and the earning opportunity cost of investing into this life insurance product instead of the risk-free bank account. Naturally, the more profit the insurer shares with their clients, the lower the dynamic surrender rate is, which implies that both $\gamma^c$ and $\gamma^d$ are positive. Thus, after taking into account profit sharing and lapses, the dynamic of the mathematical provision is
$$
dL^\v_t = (\pi_t \dot \eta^\v_t - \dot \gamma^\v_t L^\v_t) dt = \left[ (1 + \gamma^d) \pi_t \dot \eta^\v_t - (\gamma^c + \gamma^d r) L^\v_t \right] dt. 
$$
Since the cash $\beta^\v$ is used as means for paying out claims and as residual for portfolio reallocation, its dynamic is
$$
d\beta^\v_t = \left[ \beta^\v_t r - \dot \phi_t  S_t - \dot \gamma^\v_t L^\v_t \right] dt.
$$
We also update the new quantity of the assets held in the portfolio post-reallocation with
$$d\phi^\v_t = \dot \phi_t dt.$$ 
Thus, the dynamic of the state process $X^\v=(S,\beta^\v, \phi^\v,L^\v)$ and control $\v=(\dot \phi,\pi)$ corresponds to the drift and volatility coefficients
\begin{align}\label{model : driff and vol of X}
    b(X_t^\v,\v_t) = \begin{bmatrix}
        \mu_S S_t \\ 
        \b^\v_t r - \dot{\phi_t}  S_t - \dot \gamma^\v_t L^\v_t \\
        \dot \phi_t \\
        \pi_t \dot \eta^\v_t - \dot \gamma^\v_t L^\v_t 
    \end{bmatrix}
    \text{ and }
    \sigma(X_t, \nu_t) = \begin{bmatrix}
        \sigma_S S_t\\
        0 \\
        0 \\
        0 
    \end{bmatrix}
\end{align} 
 given the definition of $\dot \eta_t$ and $\dot \gamma_t$ as in \eqref{eq: def profit intensity}-\eqref{eq: def lapse intensity}. \\

    % Cost and utility functions
The insurer seeks at controlling the expectation of the truncated difference between liabilities and total wealth at $T$: 
\begin{equation}
    G(X_T^\v) = \underline \Gc \vee \left[L^\v_T - (\phi_T^\v S_T + \beta_T^\v) \right] \wedge \overline \Gc 
\end{equation}
for some $\underline \Gc <0 < \overline \Gc$. Additionally, throughout the time horizon, we integrate some state constraints which reflect the managerical strong preferences against certain scenarios. More specifically, for any $t \in [0,T]$, we penalize for the cases of bankruptcy, of negative mathematical provision, and of negative cash respectively with the following running cost functions
\begin{align}
    g^1(X_t^\v) &= (L_t^\v - \phi_t^\v S_t - \beta_t^\v)^+ \label{constraint: no bankruptcy} \\
    g^2(X_t^\v) &= (L_t^\v)^- \label{constraint: non-negative PM} \\
    g^3(X_t^\v) &= (\beta_t^\v)^- \label{constraint: non-negative cash}
\end{align} where $x^+ := \max(x,0)$ and $x^- := \max(-x,0)$ for any $x \in \R$. We therefore define the total penalty $g$ as
$$
g := \kappa (g^1 + g^2 + g^3),
$$
for some $\kappa>0$.
Under the expectation constraint associated with $G$ and $g$, the objective is to maximise the utility of the terminal capital gains, which is defined by : 
$$
F(X_T^\nu)=\underline{\Fc} \vee -\frac1\zeta e^{-\zeta(\phi_T^\v S_T + \beta_T^\v-L^\v_T)} 
$$
where $\zeta \in (0, 1)$ is the risk aversion coefficient, and $\underline \Fc <0$ \footnote{We set $f = 0$ since there is no intrinsic accumulative reward term in this model.}.

% Heurestic analysis of the optimization problem
\subsection{Formal analysis of the optimization problem} \label{sect_analysis}
Before diving into the numerical resolution of our toy model, we would like to give some intuition on the solution of the problem, which can also be helpful later on when we want to verify the soundness of our estimations in the upcoming section. \\

    % Dynamic of wealth
For this, we  look at the dynamic of the net wealth $\phi^\nu_t S_t + \beta^\nu_t - L^\nu_t$ which is crucial to multiple parts of our optimization : in the terminal utility function $F$ to be maximized, in the terminal constraint associated to $G$, and also in the bankruptcy penalty $g^1$. In light of \eqref{eq: def lapse intensity}-\eqref{eq: def profit intensity}-\eqref{model : driff and vol of X}, we observe that 
\begin{align*}
    & d(\phi^\nu_t S_t + \beta^\nu_t - L^\nu_t)  \\
    & = \phi^\nu_t dS_t + S_td\phi^\nu_t + d\b^\nu_t - dL^\nu_t \nonumber \\
    & = \left[\phi^\nu_t \mu_S S_t +\b^\nu_t r - 
 \pi_t \left(\b^\nu_t r + (\dot \phi_t)^-(S_t - \tilde S) \right) \right]dt + \phi^\nu_t \sigma_S S_t dW_t.  \nonumber 
\end{align*}

Let us first disregard the dependence of the cash dynamics $\beta^\nu$ with respect to the lapse rate $\dot{\gamma}^\nu$. When the risky asset price is at gain, i.e.~$S\ge \tilde S$, then the drift is maximized by avoiding sells, i.e. $\dot \phi^-=0$, and by sharing at minimum the profits generated by the bank account, i.e.~$\pi=\underline{\pi}$. The other way round when losses are significant and $\beta^\nu r+|\underline{\dot \phi}| (S-\tilde S)\le 0$. When $|\underline \Gc|$ and $\overline \Gc$ are large, this  mainly drives the capacity to satisfy  the terminal constraint in expectation, while, for $|\underline \Fc|$ large, this is also an important source term for the terminal expected utility, up to the additional variance induced by $\int_0^T \phi^\nu_t \sigma_S S_t dW_t$ which is penalized by large values of $\phi^\nu$.

In practice, the situation is a little more complex as, for instance, selling stocks at losses increases the lapse intensity, see \eqref{eq: def lapse intensity}, which in turn has a negative impact on the cash process $\beta^\nu$. 

Still, from the above, one can expect that, in this very simplistic toy model and all other things being equal, the controller has a tendency to buy shares when they are at gain and keeps profits for him, and, at the opposite, sell shares when they are sufficiently at loss and then share losses with the insured. This should be more pregnant when the system is close to saturating the expectation constraint so that this one drives the all system and annihilates the  impact of the variance term induced by $\int_0^T \phi^\nu_t \sigma_S S_t dW_t$ (which, again, should play a very marginal role in the constraint in expectation for large values of $|\underline \Gc|$ and $|\overline \Gc|$).
\vspace{2mm}

  % Optimizer at extremities 
Finally note that   
$$
u\mapsto \Lc^u_X \vp  + g(\cdot, u)
$$ 
and  that (recall that $f\equiv 0$)
$$
u\mapsto \hat \mu(\cdot,u)^{\top} D_{(x,p)}\vp + \frac12 {\rm Tr}\left[(\hat\sigma\hat\sigma^{\top})(\cdot,u,\cdot)D^2_{(x,p)}\vp\right]-g(\cdot, u) D_{p}\vp
$$
are linear. 

This implies that the optimizers can be found in $\{\underline{\dot \phi},0,\overline{\dot \phi}\}\times \{\underline \pi,\overline \pi\}$. We therefore expect the optimal control $\nu$ to be of bang-bang form.

\subsection{Numerical resolution of the toy model}

    % Some information about global parameters + reference to the Annex A and Annex B 
We now apply the numerical algorithm of Section \ref{sect_algo} to the model described Section \ref{subsect_ALM model} with a time horizon discretized into $K=256$ equal periods, a market with book value of the risky asset $\tilde S = 200$ \euro{}  and a fixed interest rate $r = 3 \%$. A detailed discussion on the algorithm's implementation is deferred to Annex \ref{append : discussion on implementation}, and all model parameters as well as training hyperparameters are provided in \ref{append : parameters}. The Python code can be found at \url{https://github.com/KAPhm/stochastic_optimal_control_constraint} repository on GitHub. \\

    % Error measure
We report a total error measure of $\Ec^V = 0.069433\%$ for the joint training over a sample of size $J=10,000$. More specifically, the decomposition of the errors is $0.009776\%$ for the estimation of control $\hat \upsilon^V$ (associated with $\e^\upsilon_V$), $0.039712\%$ for the estimation of the value function before the final time (associated with $\e^V_V$), $0.009813\%$ for that at the space boundary (associated with $\e^b_V$), and $0.010131 \%$ for that at final time (associated with $\e^T_V$).\\
   
    % Commentary
In Figures \ref{fig: visual a 1} and \ref{fig: visual a 2}, we provide  illustrations of the behavior of the optimal controls as a qualitative validation of the training result. In both cases, the strategy essentially consists  in selling the stock when it is significantly at loss with respect to the book value. Losses are shared with the insured, meaning that the proportion of profit sharing is maximal in case of significant losses. On the contrary, when the price of the stock is sufficiently close to or bigger than the book value, stocks are bought at a maximum rate, and profit sharing is set at its minimum. This is consistent with all the very numerous numerical experiments we performed. Basically, the insurer keeps profits and share loses.  Obviously, this particular behavior is related to the performance criteria we used and has nothing to do with practice, but this is just a toy model. As expected, the optimal control values lie primarily
at the extremes of the permissible ranges. Numerically, it can not be exactly the case since we use smooth neural networks.\\

Additionally, comparing Figures \ref{fig: visual a 1} and \ref{fig: visual a 2} gives the impression of a greater sensitivity to price movements when the state variable gets closer to the boundary of the domain, i.e. when $(t, X_t, P_t)$ approaches $(t, X_t, \varpi(t, X_t))$. In all our numerical experiments, the optimal control is sufficiently well estimated to avoid the case where the boundary is crossed significantly (see e.g. Figure \ref{fig: visual a 2} in which the optimal path {\sl visually} follows the boundary after it reaches it). We recall that, due to the discretization of the time horizon, a crossing of the domain boundary, or so-called an overshoot, can happen. \\

All this is perfectly  consistent with the a-priori analysis of Section \ref{sect_analysis}.\\

Cognizant of the simplicity of the current ALM model, which is a toy model intended for the illustration of the numerical method introduced in Section \ref{sect_algo}, we defer the development of a more complex and realistic model and further testing or improvement of the algorithm to later research.\\

% Figures 
\begin{figure}[H]
    \centering
    \includegraphics[width=0.85\textwidth]{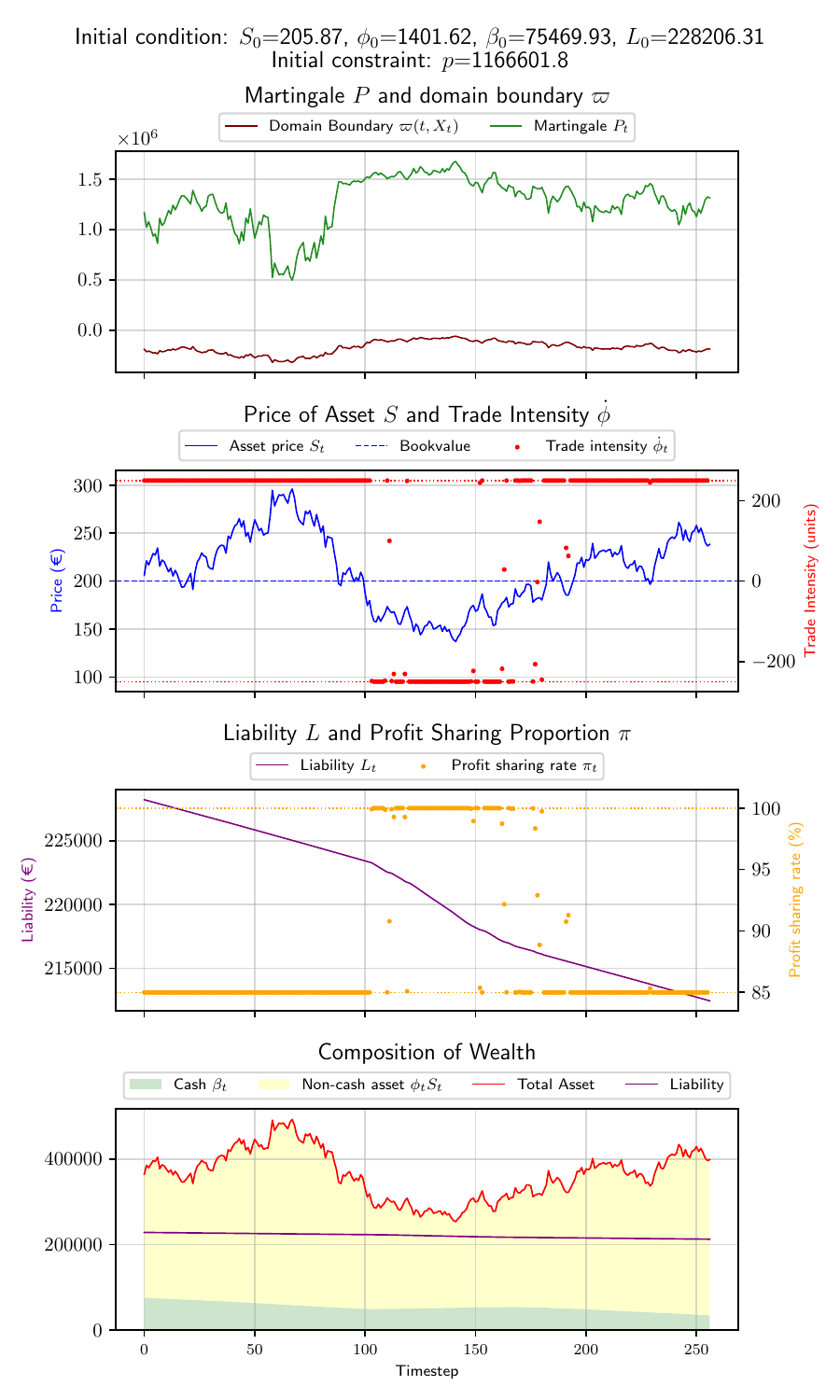}
\caption{Example of a simulated path with state variable remaining far from the boundary. When the stock price is significantly at loss, the insurer sells it and shares his loss with the insured.}
    \label{fig: visual a 1}
\end{figure}

\begin{figure}[H]
    \centering
    \includegraphics[width=0.85\textwidth]{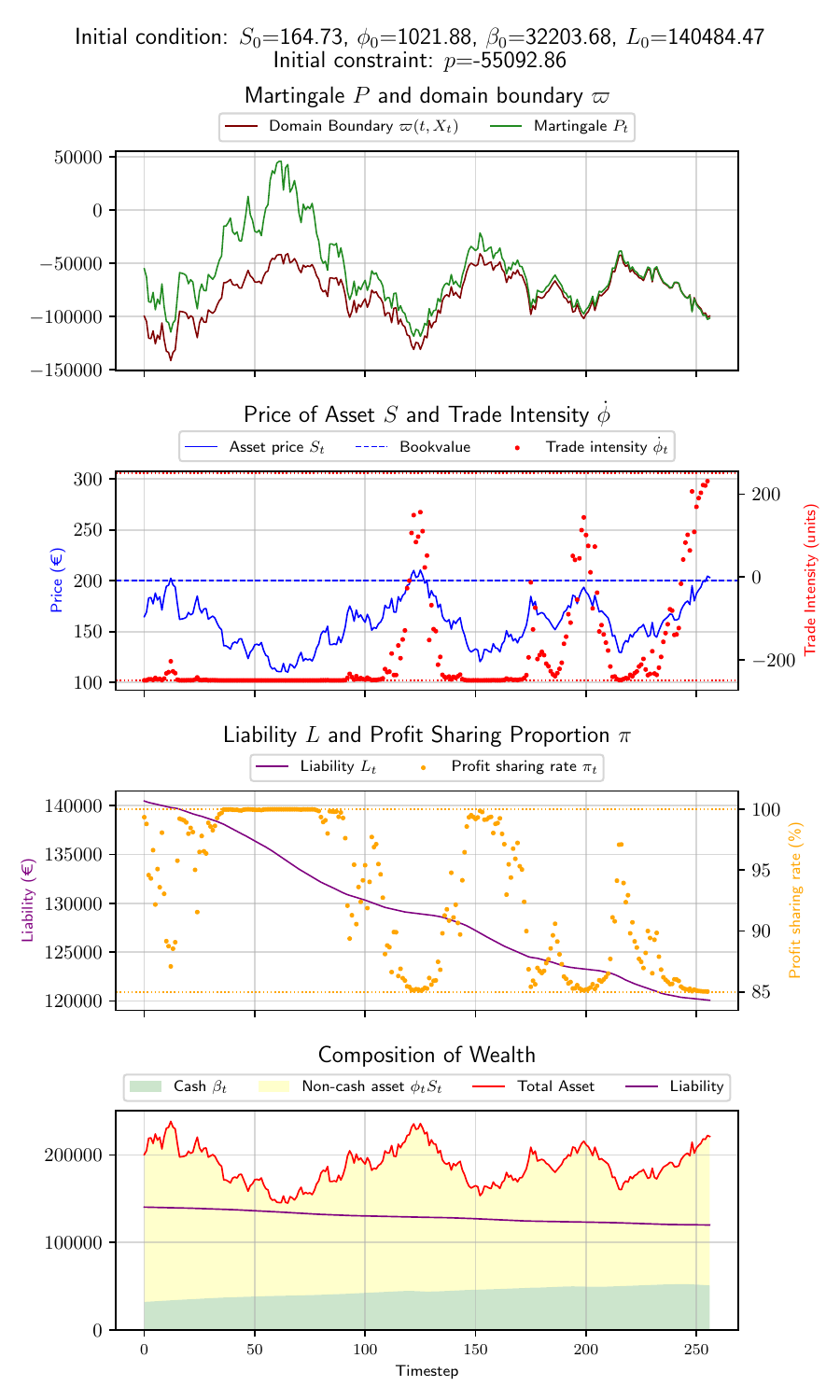}
    \caption{Example of a path reaching the boundary of the domain. Once the domain's boundary is reached (around $t_{225}$), the state process is "absorbed" by it until the end of the time horizon.}
    \label{fig: visual a 2}
\end{figure}

% ---------------------------------
%%% Proofs %%%
%----------------------------------
\section{Proofs of the PDEs characterization} \label{sect_proofs}

We provide in this section the proofs of the main results of Section \ref{sect_formulation}.
 
% Proof for smoothness of w
\subsection{Smoothness of $\varpi$}\label{sec: smoothness w}

\no The smoothness result of Proposition \ref{prop: w smooth} is expected. Since we work on an unbounded domain, we provide a self-contained proof of the required local estimates by lack of an appropriate reference.  In particular, we make use of the arguments in \cite{abeille_diffusive_2023} and \cite{andrews_fully_2004}. 

\begin{proof}[Proof of Proposition \ref{prop: w smooth}]	
	\no  Define 
	$$ 
	r(x,u):=\Lc_{X}^{u}G(x)+g(x,u),\;(x,u)\in \R^d\x U,
	$$
	and recall from Assumption \ref{assum_g_holder} that $r$ is well-defined, bounded and H\"older continuous.
	By It\^{o}'s Lemma,
	$$
	\tilde w(t,x):=\varpi(t,x)-G(x)=\inf_{u\in \Uc}\Ex{\int_{t}^{T}r(X_{t,x}^{\nu}(s),\nu_{s})ds}
	$$
	and we just need to prove the required the regularity for $\tilde w$. More specifically, we claim that $\tilde w$ is the unique smooth solution with polynomial growth of 
    \begin{equation} \label{PDE_w_til}
        \begin{split}
        \partial_t \phi + L(\cdot,D\phi,D^2\phi) &= 0 \text{ on } \domint \\
        \phi(T,\cdot) &= 0 \text{ on } \Rd
        \end{split}
    \end{equation} 
    in which 
    \begin{equation*}
		L(x,p,A) = \inf_{u \in U}  \left( \mu(x,u)^{\top}p  + \tr{\sigma \sigma^\top (x,u) A} + r(x,u) \right) \text{ for } (x,p,A) \in \Rd \times \Rd \times \Sd.
	\end{equation*}

	\no \textbf{a. Existence and smoothness on a bounded domain.} \\
    \no As in \cite{abeille_diffusive_2023}, we first restrict to a bounded domain. Recall that $B_{k}$ denotes the ball of radius $k>0$ centered at $0$ and define 
	\begin{equation}
		\tilde w_k(t,x) = \inf_{\nu \in \Uc} \Ex{\int_{t}^{T \wedge \t^{t,x,\nu}_k} r(X^{\nu}_{t,x}(s), \nu_s)ds}
	\end{equation}
	where $ \t^{t,x,\nu}_k := \inf\{s \geq t : X^{\nu}_{t,x}(s) \notin B_k\}$.   
	Then, it follows from \cite[Theorem 14.24]{lieberman_second_2005}, recall Assumption \ref{assum_sigma_unif_elliptic}, and a standard verification argument that $\tilde w_{k}$ is a smooth solution of 
	\begin{align}\label{eq: pde wk} 
	0=L(\cdot,D\tilde w_{k},D^{2}\tilde w_{k})\mbox{ on } [0,T)\x B_{k},\; \tilde w_{k}=0\mbox{ on } \partial_{P}( [0,T)\x B_{k}),
	\end{align}
	in which $\partial_P$ denotes the parabolic boundary of a subset of $[0,T]\x \R^d$.
	Upon using a standard smoothing argument, one can assume from now on that   {$L$} is continuously differentiable in all its arguments. In the following, we denote by $C>0$ a generic constant whose value may change from line to line but which does not depend on $(t,x,u)\in [0,T]\x \R^{d}\x U$ nor $k$. \\
	
	{\bf a.1. A-priori bound for first order derivatives :} \\
	\no The bound on $r$ implies that $|\tilde w_k(t+h,x)-\tilde w_k(t,x)| \leq C h$, for $h\in (0,T-t)$, which leads to
	\begin{align} 
		||\partial_t \tilde w_k||_\infty &\leq C \label{v_t_bounded} \\
		||\tilde w_k(t,\cdot)||_\infty &\leq C({T-t}), \;t\le T. \label{v_bounded}
	\end{align}
\no Similarly, for the $i$-th unit vector $e_i \in \Rd$ and $h\in (-1,1)$ such that $x+he_i \in B_k$, we have
	\begin{equation} \label{bound_Dv}
		\left|\tilde w_k(t,x+he_i) - \tilde w_k(t,x) \right|  \leq  
		C \sup_{\v \in \Uc} 
		 \E  \left[ \int_t^T\left|X_{t,x+he_{i}}^{\v}(s) - X_{t,x}^{\v}(s)\right| ds
		+   \left|\t^{t,x+he_i,\v}_k - \t^{t,x,\v}_k \right| \right].
	\end{equation} 
	\no By the Lipschitz continuity of $(\mu, \sigx)$,	\begin{equation} \label{diff_X}
		\Ex{\int_{t}^{T}\left|X_{t,x+he_{i}}^{\v}(s) - X_{t,x}^{\v}(s)\right|  ds} \leq C |h|.
	\end{equation}
	
	\no Furthermore, we can apply   \cite[Theorem 2.3]{bouchard_first_2017} with $\pi = 0, r=1$, and $P$  of the form $\varphi(X_{t,x+he_{i}}^{\v})$ or $\varphi(X_{t,x}^{\v})$ with $\varphi$ bounded, smooth with bounded first and second derivatives such that $\varphi(y) = |y|/k-1$ for $y$ in a neighborhood of $\partial B_{k}$. It implies that 
	\begin{equation} \label{diff_tau}
		\Ex{\left|\t^{t,x+he_i,\v}_k - \t^{t,x,\v}_k\right|}
		\leq C \Ex{\left| X_{t,x+he_i}^{\v}({\t^{t,x+he_i,\v}_k \wedge\t^{t,x,\v}_k}) - X_{t,x}^{\v}({\t^{t,x+he_i,\v}_k \wedge \t^{t,x,\v}_k}) \right|}
		\leq C|h|.
	\end{equation} 
	
	\no Combining (\ref{diff_X}) and (\ref{diff_tau}) with (\ref{bound_Dv}), we have 
	\begin{equation} \label{v_x_bounded}
		||D \tilde w_k||_\infty \leq C.
	\end{equation}

{\bf a.2. H\"older continuity of the first order derivatives :} \\
\no By linearizing  \eqref{eq: pde wk} as  in e.g.~\cite[Section 3.1]{andrews_fully_2004}, 
we obtain that, for any $0<h<T$, $(t,x)\in [0,T-h)\x B_{k}\mapsto h^{-1}(\tilde w_{k}(t+h,x)-\tilde w_{k}(t,x))$ solves an uniformly elliptic linear equation with bounded coefficients. Then, it follows  
from standard estimates deduced from the Krylov-Safonov Harnack inequality, see e.g.~\cite[Theorem 7]{andrews_fully_2004}, that there exists some  constants $C>0$ and $\b \in (0,1]$,  independent of $k$, such that for any subset $\mathcal{O} \subset [0,T)\x B_{k}$,  and any $(t,x),(t',x') \in \mathcal{O}$,
	\begin{equation*}
		|\partial_t \tilde w_k(t,x) - \partial_t \tilde w_k(t',x')| \leq C \left(|t-t'|^\frac{\b}{2} + |x-x'|^\b\right)\left(\|r\|_{\infty}+ \sup_\mathcal{O}|\partial_t \tilde w_k| \right),
	\end{equation*}
    in which $\|\cdot\|_{\infty}$ stands for the sup norm.
	Combining this with \eqref{v_t_bounded}, we obtain 
	\begin{equation} \label{v_t_Holder2}
		|\partial_t \tilde w_k(t,x) - \partial_t \tilde w_k(t',x')| \leq C \left(|t-t'|^\frac{\b}{2} + |x-x'|^\b\right), \;(t,x,t',x')\in ([0,T)\x B_{k})^{2}.
	\end{equation}

	\no Upon changing $\beta$, the same line of arguments, using now \eqref{v_x_bounded} in place of  \eqref{v_t_bounded}, lead to 
	\begin{equation} \label{v_x_Holder}
		|D\tilde w_k(t,x) - D\tilde w_k(t',x')| \leq C	\left(|t-t'|^\frac{\b}{2} + |x-x'|^{\b} \right), \;(t,x,t',x')\in ([0,T)\x B_{k})^{2}.
	\end{equation}
	
	\no \textbf{a.3. H\"older continuity of $D^2\tilde w_k$:} \\
    \no Since the regularity of $D\tilde w_k$ has been established, we treat its corresponding argument in $ {L}$ as a fixed parameter and consider the auxiliary problem  
	\begin{align*}
		\partial_t \phi + \tilde{L}(t,x, D^2\phi(t,x)) &= 0,\;(t,x) \in [0,T)\x B_{k} \\
		\phi(t,x) &= 0, \;(t,x) \in \partial_{P}([0,T)\x B_{k})
	\end{align*}
	where $\tilde{L}(t,x,A) := L(x, D\tilde w_k(t,x),A)$, $(t,x,A)\in [0,T)\x \R^{d}\x\S^{d}$.\\
	
	\no Fix an arbitrary $(t_0,x_0)$ in $[0,T)\x B_{k}$ and an arbitrary $\rho>0$ such that $t_0 + \rho^{-1} < T $ and $|x_0| + \rho < k$. We first consider the following problem
	\begin{align} 
		\partial_t \phi + \tilde{L}(\cdot, D^2\phi) &= 0 \text{ on } D_0^\rho := [0,T-\rho^{-1}) \times B_\rho(x_0) \\
		\phi &=\tilde w_k \text{ on } \partial_{P}D_0^\rho.
	\end{align}
	From now on, we borrow the notations in 	\cite[Chapter IV Section 1]{lieberman_second_2005} applied to the domain $D_0^\rho$. By \cite[Lemma 14.8 and Theorem 14.7]{lieberman_second_2005}, the above problem admits a unique $C^{1,2}(D_0^\rho)$ solution $\phi$ such that 
	\begin{equation} \label{property_14.16}
		\langle \phi \rangle^*_{2+\iota_{0}} + [\phi]^*_{2+\iota_{0}} \leq C_0(|\phi|^*_2 + b_{0})
	\end{equation}
	for some positive constant $b_{0}$, $C_{0}$, and $\iota_{0} \in (0,1]$ independent of $k$. Since $\tilde w_k$ also solves this localized problem,  the unicity of the solution implies that $\phi \equiv \tilde w_k$. Thus, \eqref{property_14.16} holds for $\tilde w_k$. 
	
	\no Recalling \eqref{v_x_bounded}-\eqref{v_x_Holder}, there exists   $\b \in (0,1)$ such that $|\tilde w_k|^*_{1+\b}$ is finite. Using \cite[Proposition 4.2]{lieberman_second_2005} and the inequality $a^\lambda b^{1-\lambda}\le a +b$, $a,b\ge 0$, $\lambda \in [0,1]$, we observe that, for  $\lambda\in (0,1)$ such that $\lambda (1+\beta)+(1-\lambda)(2+\iota_0)=2$, 
	\begin{equation*} 
		|\tilde w_k|^*_2 \leq C(|\tilde w_k|^*_{1+\b})^{\lambda} (|\tilde w_k|^*_{2+\iota_0})^{1-\lambda} \leq C\left(\e^{-\frac{1-\lambda}{\lambda}}|\tilde w_k|^*_{1+\b} + \e |\tilde w_k|^*_{2+\iota_0} \right)
	\end{equation*}
	for any $\e > 0$. Using now (\ref{property_14.16}) and the fact that $\phi \equiv \tilde w_k$   leads to
	\begin{align*}
		|\tilde w_k|^*_{2+\iota_0} 
		&= \left(|\tilde w_k|^*_2 - [\tilde w_k]^*_2 - \langle \tilde w_k\rangle^*_2 \right)+ [\tilde w_k]^*_{2+\iota_0} + \langle \tilde w_k \rangle^*_{2+\iota_0} \\
		& \leq |\tilde w_k|^*_2 + C_0(|\tilde w_k|^*_2 + b_{0}). 
	\end{align*}
	Combining the above, we obtain, upon changing $C_0$ and $b_{0}$ if necessary,  
	\begin{equation*}
		|\tilde w_k|^*_{2+\iota_0} \leq 2C_0(|\tilde w_k|^*_2 + b_{0}) \leq 2C_0 \left[C \left( \e^{-\frac{1-\lambda}{\lambda}}|\tilde w_k|^*_{1+\b} + \e|\tilde w_k|^*_{2+\iota_0}\right) + b_{0} \right].
	\end{equation*} 
	Choosing $\e>0$   small enough, we deduce that there exists some constant $K_{0}$ independent of $k$ such that $|\tilde w_k|^*_{2+\iota_0} \leq K_0(1+|\tilde w_k|^*_{1+\b})$, and thus, $D^2\tilde w_k$ is locally H\"older continuous since $|\tilde w_k|^*_{1+\beta}$ is finite.\\
	
	\no \textbf{b. Convergence and smoothness of $w$.} \\
    For any $k \in \N$, we have shown that $|\tilde w_k|^*_{2+\iota} \leq C$ on $D^{\rho}_{0}$ for some constants $\iota$ and $C$ independent of $k$. Sending $k\to \infty$, it follows from the Arzel\'a-Ascoli theorem that $(\tilde w_k)_{k\ge 0}$ converges on $D^{\rho}_{0}$ to some $\tilde w_{\infty}$ such that $|\tilde w_\infty|^*_{2+\iota} \leq C$ on $D^{\rho}_{0}$. Clearly, the limit does not depend on $(t_{0},x_{0})$ nor $\rho$. Hence, convergence holds on $[0,T)\x \R^{d}$ and $\tilde w_{\infty}$ belong to $C^{1,2}([0,T)\x \R^{d})$ with locally uniformly H\"older continuous derivatives $(\partial_{t}\tilde w_{\infty},D\tilde w_{\infty},D^{2}\tilde w_{\infty})$ on $[0,T)\x \R^{d}$. It also satisfies \eqref{v_bounded}, \eqref{v_x_bounded} and solves \eqref{eq: pde wk} on $[0,T)\x \R^{d}$. By a standard verification argument, $\tilde w_{\infty}=\tilde w$, which concludes the proof of Proposition \ref{prop: w smooth}. 
	\end{proof}

% PDE characterization of value function Vb on the boundary
\subsection{PDE characterization on the boundary $\partial\Dc$} \label{sect_proof_PDE_Vc}

We provide here the proof of Theorem \ref{thm : dirichlet condition}. Although the result looks similar to the one  of \cite[Proposition 3.2]{bouchard_optimal_2010}, it is of very different nature because their Assumption 3.3.(ii) can not be satisfied in our setting:  their space  ${\cal N}_{0}(x,y,p,q)=\{(u,\sigma(x,u)q),\;u\in U\}$ is not a singleton. We also do not need their Assumption 3.4. On the other hand, we shall make use of the fact that $p\mapsto V(\cdot,p)$ is  non-decreasing.  

 \begin{Remark}\label{rem : ecriture preuve en terme de Vc} For ease of exposition, we shall write the proof of  Theorem \ref{thm : dirichlet condition} in terms of $(t,z)\in [0,T]\x \Rdd\mapsto \Vc_*(t,z):=V_*(t,z,w(t,z))$ and $(t,z)\in [0,T]\x \Rdd\mapsto \Vc^*(t,z):=V^*(t,z,w(t,z))$ in which $V^*$ and $V_*$ are the upper- and lower-semicontinuous enveloppes of $V$. It follows from \eqref{eq: def varpi}-\eqref{eq : def V_bar} that $\bar \Vc^*(t,x)= \Vc^*(t,x,y)$ and $\bar \Vc_*(t,x)= \Vc_*(t,x,y)$ so that $\Vc^*$ and $\Vc_*$ do indeed not depend on their last argument, and the associated PDEs coincides. 

In particular, $\phi$ is a test function for $\bar \Vc_*$ at $(t_0,x_0)$ if and only $(t,x,y)\mapsto \phi(t,x)$ it is a test function for $ \Vc_*$ at $(t_0,x_0,y_0)$ for any $y_0\in \R$. Obviously, it satisfies $D_y\phi=0$. The same holds for  $\bar \Vc^*$.
 \end{Remark}
 
\begin{proof}[Proof of  Theorem \ref{thm : dirichlet condition}.]

We  decompose the proof in several steps. \\ 
  
{\bf a. Supersolution property.} Fix $(t_0,x_0,y_0) \in [0,T]\x \R^{d+1}$, $z_{0}:=(x_{0},y_{0})$. Recall Remark \ref{rem : Vc decroi en p} and let $(t_{n},x_{n},y_n)_{n\ge 1}$ be a sequence in $[0,T]\x \R^{d+1}$ that converges to $(t_{0},z_0)$ and such that 
	$$
	V(t_{n},z_n,w(t_{n},z_{n}))\to \Vc_{*}(t_{0},z_0)\mbox{ as } n\to \infty,
	$$
where $z_{n}:=(x_{n},y_{n})$.\\

{\bf a.1.} We first assume that $t_{0}<T$. 	
Let $\phi$ be a $C^{1,2}_{b}([0,T]\x \R^{d+1})$ function such that $(t_0,z_0)$ achieves a  minimum of $\Vc_{*}-\phi$, equal to $0$, on $ {[0,T]\x \R^{d+1}}$.  Set 
$$
	\gamma_{n}:=V(t_{n},z_n,w(t_{n},z_{n}))-\phi(t_{n},z_n), \;n\ge 1.
	$$
 Fix $u_{0}\in \Ur(t_{0},x_{0})$ and let $\hat \ur$ and  $\hat X^{n}:=\hat X_{t_{n},x_{n}}$ be as in Assumption  \ref{assum_selection_unique_opti} with respect to $(u_{0},t_{0},x_{0},y_0)$. Set $\hat Y^n:=y_n+\int_{t_n}^\cdot g(\hat X^n(s),\hat \ur(s,\hat X^n(s)))ds$ and  let $M^{n}:=$ $w(t_{n},z_n)$ $+$ $\int_{t_{n}}^{\cdot} D\varpi(s,\hat X^{n}(s)) {\sigma(\hat X^n(s),\hat \ur(s,\hat X^n(s)))}dW_{s}$. Then, \eqref{eq: def varpi}, Proposition \ref{prop: w smooth} and 
Assumption  \ref{assum_selection_unique_opti} imply that $M^{n}$ equals $w(\cdot, \hat X^{n},\hat Y^n)$ on $[t_{n},T]$. In particular $M^{n}$ is a martingale such that $M^{n}(T)=G(\hat X^{n}(T))+\hat Y^n(T)$. Thus, by the dynamic programming principle, see  \cite[Theorem 6.1 and Remark 6.1]{bouchard_optimal_2010}, 
\begin{align*} 
&\phi(t_{n},z_n)+\gamma_{n}\\
&=V(t_{n},z_n,w(t_{n},z_{n}))
\\
&\ge \E[V_*(t_{n}+h_{n},\hat Z^n(t_{n}+h_{n}),w(t_{n}+h_{n},\hat Z^{n}(t_{n}+h_{n})))+\int_{t_n}^{t_n+h_n} f(\hat X^n(s),\hat\ur(s,\hat X^n(s)))ds]
\\
& = \E[\Vc_{*}(t_{n}+h_{n},\hat Z^n(t_{n}+h_{n}))+\int_{t_n}^{t_n+h_n} f(\hat X^n(s),\hat\ur(s,\hat X^n(s)))ds]\\
&\ge \E[\phi(t_{n}+h_{n},\hat Z^n(t_{n}+h_{n}))+\int_{t_n}^{t_n+h_n} f(\hat X^n(s),\hat\ur(s,\hat X^n(s)))ds]
\end{align*}
in which $\hat Z^{n}:=(\hat X^{n},\hat Y^{n})$, and  $h_{n}:=\sqrt{\gamma_{n}}<T-t_{n}$ upon taking $n$ large enough. Since $\hat \ur$ is continuous at $(t_{0},x_{0})$, the above implies that 
\begin{align*}
0&=\lim_{n\to \infty} \sqrt{\gamma_{n}}
\\
&\ge \liminf_{n\to \infty}\Ex{\frac{1}{h_{n}}\int_{t_{n}}^{t_{n}+h_{n}} \brak{\Lc^{\hat \ur(s,\hat X^{n}(s))}_X+g(\hat X^n(s),\hat \ur(s,\hat X^n(s)))D_y}\phi(s,\hat X^{n}(s),\hat Y^n(s))   ds}\\
&+\liminf_{n\to \infty}\Ex{\frac{1}{h_{n}}\int_{t_n}^{t_n+h_n} f(\hat X^n(s),\hat\ur(s,\hat X^n(s)))ds}
\\
&=\brak{\Lc^{u_{0}}_X+g(x_0,u_0)D_y}\phi(t_{0},z_0)+f(x_0,u_0).
 \end{align*}
It remains to remind Remark \ref{rem : ecriture preuve en terme de Vc}.\\

{\bf a.2.} We now assume that $t_{0}=T$. If $\#\{n\ge 1: t_{n}=T\}=\infty$, then we can assume that $t_{n}=T$, and therefore $V(t_{n},z_n,p_{n})=F(x_{n})$, for all $n\ge 1$. Then, $\Vc_{*}(t_{0},z_0)=F(x_{0})$ by continuity of $F$. Otherwise, we   can assume that $t_{n}<T$ for all $n$. In this case, fix $u_{n}\in \Ur(t_{n},x_{n})$ and let $\hat \ur$ and  $\hat X^{n}:=\hat X_{t_{n},x_{n}}$ be as in Assumption  \ref{assum_selection_unique_opti} with respect to $(u_{n},t_{n},x_{n})$. Arguing as above, 
$
V(t_{n},z_n,p_{n})\ge \E[F(\hat X^{n}(T))+\int_{t_n}^Tf(\hat X^{n}(s),\hat \ur(s,\hat X^{n}(s))ds]$. Then, $\Vc_{*}(t_{0},z_0)\ge F(x_{0})$ by using Remark \ref{rem : moments X}, the continuity and polynomial growth of $F$, and by taking the limit in the previous inequality, after possibly passing to a subsequence. We conclude by using  Remark \ref{rem : ecriture preuve en terme de Vc} that $\bar \Vc_{*}(t_{0},x_0)\ge F(x_{0})$.
 \\

{\bf b. Subsolution property.} 
Fix $(t_0,x_0,y_0) \in [0,T]\x \R^{d+1}$ and set $z_0:=(x_0,y_0)$.
Let $(  t_{n},  x_{n}, y_n,  p_{n})_{n\ge 1}$ in ${\rm cl}\Dc$ be converging to $(t_{0},x_0,y_0,w(t_{0},z_{0}))$ such that $\{V^{*}(  t_{n},  z_n,  p_{n})\}_{n\ge 1}$ converges to 
$\Vc^{*}(t_{0},z_0)$, in which $z_n:=(x_n,y_n)$.  Set $\gamma_{n}= p_{n}-w(  t_{n},  z_n)$. By Remark \ref{rem : Vc decroi en p}, one can assume that $\gamma_{n}>0$ for all $n$.\\

{\bf b.1.} We first assume that $t_{0}<T$ and that $p_0=w(t_0,z_0)$. 
Let $\phi$ be a $C^{1,2}_{b}([0,T]\x \R^{d+1})$ function such that $(t_0,x_0,y_0)$ achieves a  maximum of $\Vc^{*}-\phi$, equal to $0$, on $\domd$.  Let $\ell\ge 1$ be such that $V$ has at most polynomial growth of order $\ell$, and define 
\begin{align*}
\Uptheta_n (t,x,y,p) &:= V^*(t,x,y,p) - \left(\phi(t,x,y) + \gamma_{n}^{-\frac12} (p-w(t,x,y)) + \Psi(t,x,y)\right),\;
\end{align*}
for any $(t,x,y,p)\in [0,T]\x \R^{d+2},\;n\ge 1$ where $\Psi(t,x,y) := |t-t_0|^{2\ell} + |x-x_0|^{4\ell}+ |y-y_0|^{2\ell}$. 

Consider a compact neighborhood $\Oc := B_r((t_0, z_0, p_0))$ such that $t_0 + r < T$ and $\Oc \cap \inte\Dc \neq \varnothing$. Let us consider a sequence $(\hat t_{n},\hat x_{n},\hat y_n,\hat p_{n})_{n \geq 1}$ in $\Oc \cap cl\Dc$ such that  $(\hat t_{n},\hat z_n,$ $\hat p_{n})$ is a maximum point of $\Uptheta_n$ on ${\rm cl}\Dc$, in which $\hat z_n:=(\hat x_{n},\hat y_n)$. Standard arguments imply that $(\hat t_{n},\hat z_n,\hat p_{n})\to (t_{0},z_0,w(t_{0},z_{0}))$, after possibly passing to a subsequence. Indeed, up to a subsequence, it must converge to some $(t_{\infty},z_{\infty},p_{\infty})$, and we observe that 
\begin{align} 
0   &= (\Vc^* - \phi)(t_0,z_0) \nonumber \\
    &= \lim_{n\to \infty} \Uptheta_n ( t_{n}, z_n, p_{n}) \nonumber\\
    & \leq \limsup_{n\to \infty}   \Uptheta_n (\hat t_{n},\hat z_n,\hat p_{n}) \nonumber \\
    &\le  (V^* - \phi)(t_{\infty},z_\infty, p_\infty) - \lim_{n\to \infty}\left[ \gamma_{n}^{-\frac12} (\hat{p}_n-w (\hat t_{n},\hat z_n)) + \Psi(\hat t_n, \hat x_n, \hat y_n)\right]. \label{eq proof: inequality for subsolution}
\end{align}
Since $\gamma_n \downarrow 0$ and $\hat p_n \geq w(\hat t_n, \hat z_n)$ for all $n\ge 1$, it must be that $p_\infty = w(t_\infty, z_\infty)$. Consequentially, \eqref{eq proof: inequality for subsolution} leads to
\begin{align*} 
0   &= (\Vc^* - \phi)(t_0,z_0) \nonumber \\
    &\le  (V^* - \phi)(t_{\infty},z_\infty, w(t_\infty, z_\infty))\nonumber - \lim_{n\to \infty}\left[ \gamma_{n}^{-\frac12} (\hat{p}_n-w (\hat t_{n},\hat z_n)) + \Psi(\hat t_n, \hat x_n, \hat y_n)\right]\\
    & =(\Vc^* - \phi)(t_\infty,z_\infty)- \Psi(t_\infty, x_\infty, y_\infty) \\
    & \leq (\Vc^* - \phi)(t_0, z_0).
\end{align*}

We can now apply Theorem \ref{thm : pde interior domain} and \eqref{eq: def varpi} to deduce that 
	\begin{align*}  
		 &\sup_{u \in U} \crl{\Lc^u_X \phi(\hat t_{n},\hat z_n)+ g(\hat x_n,u)D_{y}\phi(\hat t_{n},\hat z_n) + f(\hat x_n,u) -  \gamma_{n}^{-\frac12}(\Lc^u_X \varpi\txnhat+ g(\hat x_n,u))} \\
		 &\ge \e_{n}	
		 \end{align*} 
	for some $(\e_n)_{n\ge 1}$ that converges to 0. Letting $n\to \infty$, recalling Proposition \ref{prop: w smooth} and the fact that  $ \gamma_{n}^{-\frac12}\to \infty$, and using the fact that $U$ is compact and that $\mu$ and $\sigma$ are continuous, we deduce that
	$$
	\Lc^u_X \phi(t_{0},z_0)+g(x_{0},u)D_{y}\phi(t_{0},z_0)+f(x_0,u)\ge 0
	$$
for some $u\in U$ such that $\Lc^u_X \varpi(t_{0},x_{0})+g(x_0,u)=0$, i.e.~$u\in U(t_{0},x_{0})$. Again, we conclude by reminding Remark \ref{rem : ecriture preuve en terme de Vc}.\\

{\bf b.2.} We now assume that $t_{0}=T$. By arguing as in a.2.~it suffices to consider the case where $t_{n}<T$ for all $n\ge 1$. Since $p_{n}>w(t_{n},x_{n},y_n)$, we can find $\nu_{n}\in \Uc$ such that 
$\E[{\rm G}^{\nu_{n}}_{t_{n},x_{n},y_n}]\le p_{n}$ and $V(t_{n},x_{n},y_n,p_{n})\le \E[{\rm F}^{\nu_{n}}_{t_{n},x_{n}}]+n^{-1}$, for each $n\ge 1$. By  Remark \ref{rem : moments X}, the continuity and polynomial growth of $f$, passing to the limit, possibly along a subsequence, implies that $\Vc^{*}(t_{0},x_{0},y_0)\le F(x_{0})$. In view of Remark \ref{rem : ecriture preuve en terme de Vc}, $\bar \Vc^{*}(t_{0},x_{0},y_0)\le F(x_{0})$.
\end{proof}

\begin{Remark}
    Note that in the above proof, the control $\a$ for the Martingale representation process $M$ does not appear when we apply the result of Theorem \ref{thm : pde interior domain} due to the fact that the test function  $\phi$ does not depend on $p$. 
\end{Remark}

\appendix 

% ---------------------------------
%%% Appendix : Implementation techniques %%%
%----------------------------------
\section{Further discussion on implementation of the numerical resolution} \label{append : discussion on implementation}

In Section \ref{sect_algo}, we have introduced a multi-step algorithm used for solving the stochastic optimal control under constraints in its essence. Beyond the method described in \ref{sect_algo}, we also use some practical suggestions from the growing machine learning literature to fine-tune our PINN training with certain adaptations to the specificity of our problem. \\

% Pretraining : scaling
First and foremost, as a good practice, we scale our inputs so that all variables have similar magnitude in order to avoid biases against variables with inherently lower magnitude value such as interest rate. Furthermore, as discussed in  \cite[Section 3]{wang_2023}, this helps to improve convergence, and it also works better with common initialization schemes such as Glorot. \\

% Architecture of neural network : ResNet +  Swish + Sinusoidal ime encoding
As for the architecture of the neural networks, for those to estimate $\varpi, \Vc$, and $V$ using PINN, we use the residual block neural network (also known as ResNet) which has been shown in \cite{cheng_2021} to enhance the accuracy of PINN training and generally accepted as a better architecture for PINN training than simple Multi-Layer Perceptron (MLP) as discussed in \cite{wang_2023}. Additionally, we opt for Swish activation function as described in \cite{ramachandran_swish_2017} as our activation function instead of the commonly used Rectified Linear Units (ReLU). We also encode the time variable with the sinusoidal encoding suggested in \cite{vaswani_attention_2023}. \\ 

Otherwise, for the neural networks to estimate the optimal control, we use simple Multi-Layer Perceptron (MLP) with sinusoidal encoding for the time variable. Furthermore, the activation function is Swish for the hidden layers whereas for the output layer, so as to ensure that the output, i.e the control variable, is within the permissible range, we use the Tanh activation function. \\ 

% Training : putting weights on different components of the loss function 
During the PINN training process, we employ a moving-weight scheme to balance between different components of the loss function, namely the PDE loss and the boundary/terminal loss. In particular, we use a deterministic monotonic weighting scheme which depends on the training time proxied by the number of epochs. For example, for the estimation of $\varpi$, at epoch $e \in \{1,...E\}$, the loss function that we actually use is
\begin{align*}
    \mathbf L^\varpi(\theta) = \lambda^\varpi_e \frac1J &\sum^J_{j=1} \frac1K \sum^{K-1}_{k=0} \left| \Lc^{\vhat^{\varpi,j}\tk}_X n^\varpi_\t (t_k, X^{\vhat^\varpi,j}\tk)+g(X^{\vhat^\varpi,j}\tk, \vhat^{\varpi, j}\tk)\right|^2 \\
    & + \frac1J \sum^J_{j=1}  \left| n^\varpi_\t(T, X^{\vhat^\varpi,j}_T) - G(X^{\vhat^\varpi,j }_T)\right| ^2
\end{align*}
with 
\begin{equation} \label{append : weighting interior loss}
\lambda^\varpi_e := \underline \lambda^\varpi \vee \left[\indi{\frac{e}{d_e} \in \N}\left(\lambda^\varpi_0 + \frac{(\lambda^\varpi_E - \lambda^\varpi_0)*e}{E} \right) + \indi{\frac{e}{d_e} \notin \N}\lambda^\varpi_{e-1} \right] \wedge \bar \lambda^\varpi
\end{equation}
where $d_e$ is the frequency of weight update, $E \geq 1$ is the number of epochs beyond which the weight will no longer be updated, $\lambda^\varpi_0$  and $\lambda_E^\varpi$ are initial and final values for $\lambda^\varpi_e$, and $0 < \underline \lambda^\varpi < \bar \lambda^\varpi$ are lower and upper bounds for $\lambda^\varpi_e$. In other words, starting from $\lambda_0^\varpi$ at the first epoch until reaching $\lambda^\varpi_E$ at the $E$-th epoch, the loss weight $\lambda^\varpi$ gets updated after every $d_e$ epochs in a linear manner while being bounded within $[\underline \lambda^\varpi, \bar \lambda^\varpi]$. All of these are considered hyper-parameters for training, and we apply similar weighting scheme for the other training that requires PINN. During the training process, we have also experimented with more sophisticated schemes such as the self-adaptive weighting for each data point described in \cite{mcclenny_2023} or the dynamically-adjusted loss weighting using $\mathbb L^2$ norm as in \cite{wang_2023}, but, for our specific problem, the deterministic scheme provides the best training in terms of both accuracy and training stability.\\

% Training : optimizer + learning rate adaptation + batching
We use Adam optimizer with weight decay and combine an exponential and a multi-step decay schedulers to fine-tune our learning rate. Finally, we divide our data sample into batches for each iteration/epoch. The gradient descend is performed at the batch level whereas the learning rate is updated at the epoch level. Since the PINN method relies heavily on the computation of gradients, batching effectively reduces the amount of memory stockage needed for each gradient descend and hence, speeds up the training time. 

% Parameters for all 
\section{Parameters for the application example} \label{append : parameters}

We specify in the following the parameters that we have used for the numerical demonstration in Section \ref{sect_toy example} including the parameters for the ALM model and the hyper-parameters for the neural network trainings. 

    % Parameters for the model (including the distribution for sampling)
\subsection{Parameters for the ALM model} 
        % parameters for dynamic
We recall that the time horizon $T = 1$ is discretized into $K=256$ periods. For modeling the assets, we use $\mu_S = 0.15, \sigma_S = 0.25$ with fixed interest rate $r=3\%$ and fixed book value $\tilde S = 200$ \euro{}. For modeling surrenders, the cyclical lapse rate $\gamma^c = 5\%$ and the dynamic lapse coefficient $\gamma^d = 0.5$. \\

        % parameters for sampling 
At the beginning of the time horizon $t_0 = 0$, the initial values for the risky asset price $S_0$, the quantity of risky asset $\phi_0$, the cash-to-non-cash ratio $\delta^\b_0$, and the liability-to-total-asset ratio $\delta^L_0$ are drawn from a uniform distribution with the ranges given in Table \ref{table : param sampling}. Then, we calculate initial cash and initial liability as 
$$
\beta_0 = \delta^\b_0 \phi_0 S_0 \text{ and } L_0 = \delta^L_0(\beta_0 + \phi_0 S_0).
$$
The reasoning for sampling cash-to-non-cash ratio and liability-to-total-asset ratio instead of directly sampling cash and liability is to guarantee that liability is inferior to total asset (i.e the portfolio is not insolvent) and that the majority of asset value is placed in the risky asset, which is realistic. Note that, as discussed in Appendix \ref{append : discussion on implementation}, we apply scaling to the data sample before training neural network.\\

\begin{table}[H]
\centering
\begin{tabular}[h]{|l|c|c|c|}
    \hline
    State variables & Interval pre-scaling & Interval post-scaling & Scaling factor\\
    \hline
    Risky asset price - $S_0$ & [150.0, 250.0] & [0.15, 0.25] & $10^{-3}$ \\
    Quantity asset - $\phi_0$ & [1000.0, 2000.0] & [1.0, 2.0] & $10^{-3}$ \\
    Cash-to-non-cash ratio - $\delta^\b_0$ & [0.1, 0.4] & [0.1, 0.4] & $1$ \\
    Liability-to-total-asset ratio - $\delta^L_0$ & [0.6, 0.8] & [0.6, 0.8] & 1 \\
    \hline
\end{tabular}
\caption{Distribution of state variables at $t=0$}
\label{table : param sampling}
\end{table}

        % parameters for control variables
As for the control variables, we allow the controls for the state variable and for the Martingale representation to be within the bounds as in Table \ref{table : control bounds}. Furthermore, the limits for the cost function $G$ and the utility function $F$ can also be found in Table \ref{table : control bounds}. 

\begin{table}[H]
    \centering
    \begin{tabular}{|l|c|c|c|}
    \hline
    Control variables & Interval pre-scaling & Interval post-scaling & Scaling factor \\
    \hline
    Trade intensity - $\dot \phi$ & [-250.0, 250.0] & [-0.25, 0.25] & $10^{-3}$ \\ 
    Profit sharing rate - $\pi$ & [0.85,1.00] & [0.85, 1.00] & 1 \\
    
    \hline
    \hline 
    Bounds & Value pre-scaling & Value post-scaling & Scaling factor \\
    \hline
    Minimal terminal capital loss $\underline \Gc$ & $-10^{13}$ & $- 10^{10}$ & $10^{-3}$ \\
    Maximal terminal capital loss $\overline \Gc$ & $10^{13}$ & $10^{10}$ & $10^{-3}$ \\
    Minimal terminal utility $\underline \Fc$ & $-10^{10}$ & $-10^{10}$ & 1 \\
    Limit $\Nc$ for Martingale controls $\a$ & $10^5$ & 10 & $10^{-6}$ \\ 
    Penalization coefficient $\kappa$ &16 &16& 1\\
    \hline
    \end{tabular}
    \caption{Bounds for controls and functions}
    \label{table : control bounds}
\end{table}

The risk aversion coefficient for the exponential utility function $F$ is set to be $\zeta = 0.25$. The sampling parameter $\Ec_P$, which is used in the sampling of the initial constraint $P_0$ as described in \eqref{sampling P},  is set to 10.0 post-scaling, which is equivalent to $10^5$ pre-scaling \footnote{$P^\a$ is the Martingale representation of the terminal capital loss $L_T - \b_T - \phi_T S_T$, which has the same magnitude as cash $\beta$, so $P$ shares the same scaling factor as cash $\b$. Since brownian $W$ is unit less, $\a$ has the same scaling factor as $P$.}.

    % Hyperparameters for training neural networks
\subsection{Hyperparameters for training neural networks} 

We recall that we redraw the initial states and brownian paths for each training step and for the validation step with the intention of maximizing the generality of the neural networks. With the given portfolio model with one risky asset, the brownian (representing the risk factor) is 1-dimensional. In the following, $J$ represents the sample size, i.e the number of initial states to be drawn. We represent the hyper-parameters for the control neural networks $\hat \v^\varpi$ and $\hat \upsilon^V$ together since they have the same architecture and the same training method, and similarly for the neural networks $\hat \varpi, \hat \Vc$, and $\hat V$. \\

\textbf{Parameters for training $\hat \v^\varpi$ and $\hat \upsilon^V$}\\

For the training of $\hat \v^\varpi$, the milestones for the multi-step scheduler (for updating the learning rate) are $[500, 1000, 1500, 2000, 2500, 3000, 3500, 4000]$, and for the training of $\hat \upsilon^V$, they are $[500, 1500, 2500, 3500, 4500, 5500, 6500, 7500, 8500, 9500]$.

Specifically for training the optimal control on the domain boundary, we add a weight $\lambda_g = 16.0$ to the intermediary cost function $g$ to guide the training to also pay attention to minimize the running costs. This weight does not interfere with other trainings. \\

\begin{table}[H]
\centering
\begin{tabular}{|l|l|c|c|} 
    \hline
    \multicolumn{2}{|c|}{Neural Network} & $\vhat^\varpi$ & $\hat \upsilon^V$\\
    \hline
    \multirow{3}{*}{Structure} & Number of frequencies (for time encoder)& 8     & 8 \\
        & Number of hidden layers  & 6 & 7 \\
        & Number of neurons per hidden layer & 100 & 128 \\
    \hline
    \multirow{3}{*}{Training sample} & Number of training epoch & 5,000 & 10,000\\
        & Total sample size $J$ & 50,000 & 50,000 \\
        & Batch size & 10,000 & 10,000 \\
    \hline
    \multirow{4}{*}{Scheduler and Optimizer} & Initial learning rate & 0.005 & 0.0025\\
        & Exponential scheduler decay coefficient & 0.9996& 0.9996\\
        & Multi-step scheduler decay coefficient & 0.6 & 0.6 \\
        & Weight decay coefficient for Adam optimizer & $10^{-7}$ & $10^{-7}$ \\
    \hline
\end{tabular} 
\caption{Training hyper-parameters for $\vhat^\varpi$ and $\hat \upsilon^V$}
\label{table : param traing u and a}
\end{table}

% - Hyperparameters for training w, Vb and V
\textbf{Parameters for training $\hat\varpi$, $\hat\Vc$, and $\hat V$} \\
The milestones for the multi-step scheduler for the training of 
\begin{itemize}
    \item $\hat \varpi$ : [1000, 2000, 3000, 4000, 4500, 5000, 5500, 6000, 6500, 7000, 7500]
    \item $\hat \Vc$ : [1000, 2000, 3000, 4000, 4500, 5000, 5500, 6000]
    \item $\hat V$ : [1500, 3000, 4500, 6000, 7500]
\end{itemize}

\begin{table}[H]
    \centering
    \begin{tabular}{|l|l|c|c|c|}
        \hline
        \multicolumn{2}{|c|}{Neural Network} & $\hat \varpi$ & $\hat \Vc$ & $\hat V$ \\
        \hline 
        \multirow{3}{*}{Structure} & Number of frequencies (for time encoder)& 8     & 8  & 8 \\
            & Number of hidden layers  & 6 & 6 & 8  \\
            & Number of neurons per hidden layer & 128 & 128 & 100 \\
        \hline 
        \multirow{3}{*}{Training sample} & Number of training epoch & 10,000 & 6,000 & 8,000\\
            & Total sample size $J$ & 5,000 & 5,000 & 5,000\\
            & Batch size & 250 & 250 & 250 \\
        \hline
        \multirow{4}{*}{Scheduler and Optimizer} & Initial learning rate & 0.001 & 0.002 & 0.004 \\
            & Exponential scheduler decay coefficient & 0.9996 & 0.9996 & 0.998\\
            & Multi-step scheduler decay coefficient & 0.5 & 0.75 & 0.5 \\
            & Weight decay coefficient for Adam optimizer & $10^{-7}$ & $10^{-7}$ & $10^{-7}$\\
        \hline
        \multirow{4}{*}{Weight for PDE loss $\lambda$ \footnote{Recall the formulation in \eqref{append : weighting interior loss}.}} & Initial weight for interior loss $\lambda_0$ & 0.0 & 0.0 & 0.0\\
            & Final weight for interior loss $\lambda_E$ & 0.5 & 0.5 & 0.5\\
            & Frequency of weight update $d_e$  & 500 & 500 & 750\\
            & Last epoch of weight update $E$ & 5000 & 5000 & 7500 \\
        \hline 
    \end{tabular}
    \caption{Training hyper-parameters for  $\hat \varpi$, $\hat \Vc$, and $\hat V$}
    \label{table : param training w, Vb, and V}
\end{table}

%% Include the bibliography
\newpage
\bibliographystyle{plain}

\end{document}